\documentclass[pdflatex,sn-mathphys-num]{sn-jnl}

\usepackage{graphicx}%
\usepackage{multirow}%
\usepackage{amsmath,amssymb,amsfonts}%
\usepackage{amsthm}%
\usepackage{mathrsfs}%
\usepackage[title]{appendix}%
\usepackage[dvipsnames]{xcolor}%
\usepackage{textcomp}%
\usepackage{manyfoot}%
\usepackage{booktabs}%
\usepackage{algorithm}%
\usepackage{algorithmicx}%
\usepackage{algpseudocode}%
\usepackage{listings}%
\usepackage{overpic}
\usepackage{cleveref}
\usepackage{enumitem}
\usepackage{tikz}

\usepackage{etoolbox}
\makeatletter
\patchcmd{\ps@headings}
    {\@oddfoot{\hfill\thepage\hfill}}  
    {\@oddfoot{\parbox[t]{.66\textwidth}{\footnotesize Distribution Statement A: Approved for Public Release; Distribution is Unlimited. PA\# AFRL-2024-4346, 08/07/2024.}\hfill\thepage}} 
    {}
    {}
    
\theoremstyle{thmstyleone}%
\newtheorem{theorem}{Theorem}
\newtheorem{lemma}{Lemma}

\theoremstyle{thmstyletwo}%

\theoremstyle{thmstylethree}%

\begin{document}

\title[High-order Contour Integral Methods for Semigroups]{A family of high-order accurate contour integral methods for strongly continuous semigroups}

\author*[1]{\fnm{Andrew} \sur{Horning}}\email{hornia3@rpi.edu}

\author[2]{\fnm{Adam R.} \sur{Gerlach}}\email{adam.gerlach.1@us.af.mil}

\affil*[1]{\orgdiv{Department of Mathematical Sciences}, \orgname{Rensselaer Polytechnic Institute}, \orgaddress{\street{1121 Sage Ave}, \city{Troy}, \postcode{12180}, \state{NY}, \country{USA}}}

\affil[2]{\orgdiv{Control Sciences Center}, \orgname{Air Force Research Laboratory}, \orgaddress{\street{2210 Eighth Street}, \city{WPAFB}, \postcode{45433}, \state{OH}, \country{USA}}}

\abstract{Exponential integrators based on contour integral representations lead to powerful numerical solvers for a variety of ODEs, PDEs, and other time-evolution equations. They are embarrassingly parallelizable and lead to global-in-time approximations that can be efficiently evaluated anywhere within a finite time horizon. However, there are core theoretical challenges that restrict their use-cases to analytic semigroups, e.g., parabolic equations. In this article, we use carefully regularized contour integral representations to construct a family of new high-order quadrature schemes for the larger, less regular, class of strongly continuous semigroups. Our algorithms are accompanied by explicit high-order error bounds and near-optimal parameter selection. We demonstrate key features of the schemes on singular first-order PDEs from Koopman operator theory.}




\maketitle

\section{Introduction}\label{sec:intro}

This paper is about computing the action of an operator exponential, i.e.,
\begin{equation}\label{eqn:op_exp}
    K(t)x \, = \, \exp(At)x, \qquad\text{for all}\qquad 0\leq t\leq T,
\end{equation}
where $x$ is a vector in a separable Banach space $\mathcal{X}$ and $A$ is a linear operator that generates a strongly continuous semigroup on $\mathcal{X}$ (see~\cref{sec:sc_semigroups}). Given $A$, we aim to approximate the action of $K(t)$ over the real interval $[0,T]$. As an infinite-dimensional analogue of the matrix exponential, $K(t)$ is central to the theory, analysis, and approximation of autonomous time-dependent processes on the Banach space $\mathcal{X}$:
\begin{equation}\label{eqn:diff_eq}
    \frac{dx}{dt} = Ax, \qquad\text{with initial condition}\qquad x|_{t=0}=x_0\in\mathcal{X}
\end{equation}
The solution of~\cref{eqn:diff_eq} is provided, formally, by $x(t) = K(t)x_0$, and the theory of strongly continuous semigroups provides rigorous constructions of $K(t)$, characterizations of well-posedness, and notions of regularity in the abstract setting~\cite{engel2000one,hille1996functional,davies1980}. In a computational context, approximating $K(t)$ is a core task of exponential integrators, a class of robust algorithms used for challenging stiff evolution equations~\cite{hochbruck2010exponential,schmelzer2007,minchev2005review}. 

Among the many techniques for computing exponentials of matrices and operators~\cite{dubious_matrix_exp,Saad1992,Lubich1997,BEYLKIN1998362,COX2002430,mclachlan2002splitting,Ostermann2005,schmelzer2007,Higham2011,dorsek2012semigroup,rousset2021general}, methods based on contour integrals possess remarkable properties. To illustrate, suppose that $A$ is a matrix and $\gamma$ is a simple Jordan curve winding once counterclockwise around the spectrum of $A$. If $w_1,\ldots,w_N$ and $z_1,\ldots,z_N$ are suitable complex quadrature weights and nodes, then one approximates~\cite[Sec.~16]{trefethen2014exponentially}
\begin{equation}\label{eqn:cont_idea}
    \exp(At)x = \frac{1}{2\pi i}\int_\gamma e^{zt}(z-A)^{-1}x\, dz
    \approx \frac{1}{2\pi i}\sum_{k=1}^N w_k e^{z_k t} (z_k-A)^{-1}x.
\end{equation}
The main computational expense is solving shifted linear equations at the quadrature nodes, after which the approximation is formed from linear combinations of the solutions. There are three particularly attractive features of this computational framework. First, simple quadrature schemes often achieve rapid convergence and, consequently, highly accurate solutions are obtained by solving a few linear systems. Second, each shifted linear equation can be solved independently of the others, making the scheme trivially parallelizable. Third, each shifted linear equation need only be solved once for a given $x\in\mathcal{X}$. Afterwords, the approximation  of $K(t)x$ can be evaluated at any time $t\geq 0$ via linear combinations of vectors in $\mathcal{X}$, with no matrix-level computations.

In contrast to the matrix case, operators that generate semigroups associated with partial differential equations and stochastic processes are generally infinite-dimensional and unbounded. The spectrum of $A$ is unbounded and analogs of~\cref{eqn:cont_idea} involve unbounded contours (see~\cref{sec:func_calc}). The key to efficient contour integral methods in this setting is the design of rapidly convergent quadrature rules for the associated improper integrals. Such approximations have been developed for classes of elliptic partial differential operators~\cite{sheen2000parallel,sheen2003parallel}, positive operators in Banach spaces~\cite{gavrilyuk2001exponentially,gavrilyuk2005exponentially},  and in related contexts~\cite{hale2015contour,hale2008computing}. These works exploit \textit{analyticity in the semigroup} generated by $A$, which allows one to deform the contour integral into the left half-plane and leverage fast quadrature rules for smooth, exponentially decaying integrands~\cite{talbot1979accurate,duffy1993numerical,weideman2006optimizing,trefethen2006talbot} (see the left panel of~\cref{fig:semigroup_geometry}). Recently, Colbrook derived the first numerically stable and near-exponentially convergent quadratures for general analytic semigroups~\cite{colbrook2022computing}.

\begin{figure}
    \centering
    \begin{minipage}{0.48\textwidth}
        \begin{overpic}[width=\textwidth]{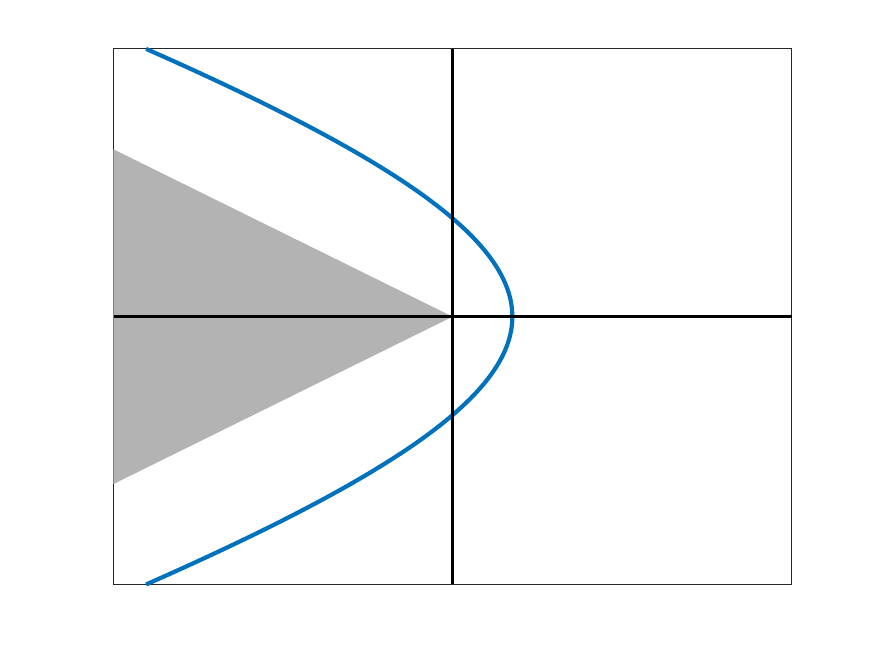}
        \put (48,1) {$\displaystyle {\rm Re}\,z$}
        \put (2,36) {\rotatebox{90}{$\displaystyle {\rm Im}\,z$}}
        \put (29,72) {Analytic Semigroup}
        \put (60,40) {rapid}
        \put (60,34) {decay}
        \put (48,46) {\tikz\draw[->,line width=1pt] (0,0) -- (-0.8,0.7);}
        \put (48,20) {\tikz\draw[->,line width=1pt] (0,0) -- (-0.8,-0.7);}
        \end{overpic}
    \end{minipage}
    \begin{minipage}{0.48\textwidth}
        \begin{overpic}[width=\textwidth]{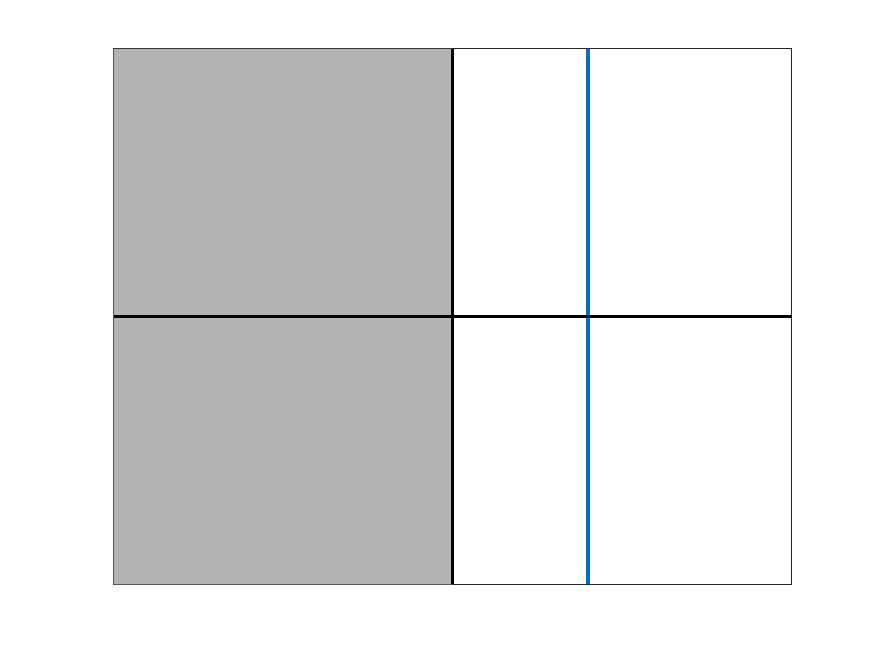}
        \put (48,1) {$\displaystyle  {\rm Re}\,z$}
        \put (2,36) {\rotatebox{90}{$\displaystyle {\rm Im}\,z$}}
        \put (13,72) {Strongly Continuous Semigroup}
        \put (70,40) {slow}
        \put (69,34) {decay}
        \put (75,46) {\tikz\draw[->,line width=1pt] (0,0) -- (0,0.75);}
        \put (75,20) {\tikz\draw[->,line width=1pt] (0,0) -- (0,-0.75);}
        \end{overpic}
    \end{minipage}
    \caption{\label{fig:semigroup_geometry} When $A$ generates an analytic semigroup, its spectrum is contained in a sector (shaded  grey) and efficient contour integral methods leverage a rapidly decaying integrand along the contour (blue line) in the left half-plane~\cite{colbrook2022computing}. For strongly continuous semigroups, the spectrum may fill the whole left half-plane and standard contour integral representations may have a slowly decaying integrand.}
\end{figure}
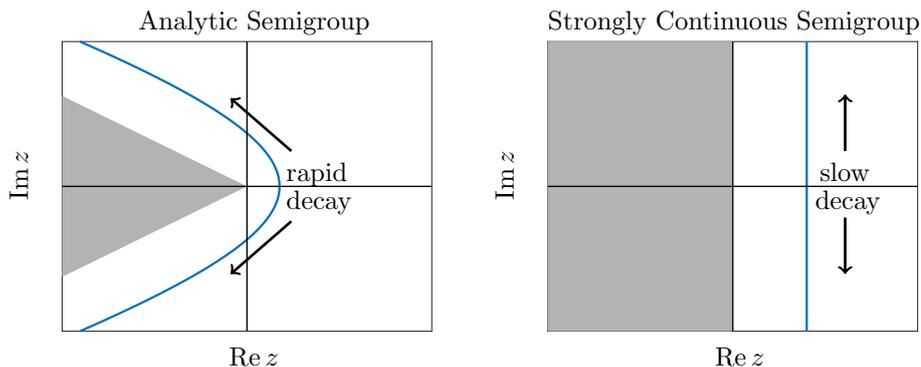

While analytic semigroups comprise an important class of parabolic-like evolution equations, many time-dependent processes associated with transport, waves, and other causal processes with limited regularity are studied in the setting of strongly continuous semigroups. However, there are theoretical obstacles to the development of fast, global, and highly-parallelizable contour integral schemes for this class. When $A$ generates a strongly continuous semigroup, its spectrum may fill the whole left half-plane and contour integral representations of $K(t)$ typically have exponentially growing integrands in the right half-plane. Consequently, contour deformation cannot provide any significant advantage (see the right panel of~\cref{fig:semigroup_geometry}). Furthermore, standard contour integral representations may not converge for all vectors $x\in\mathcal{X}$ (see~\cref{eqn:Brom_inv}). When they do converge, the integrand may decay very slowly. These features make it difficult to design effective quadrature rules for strongly continuous semigroups.

To overcome these obstacles, we propose a new family of contour integral methods for strongly continuous semigroups that exploits \textit{regularity in initial vectors} rather than the semigroup. This idea is common in time-stepping schemes for evolution equations but appears to be unexplored in the context of contour integral schemes for semigroups~\cite{colbrook2022computing}. Therefore, we make the following contributions:
\begin{itemize}
\item We show how to use a regularized functional analysis for strongly continuous semigroups to accelerate contour integral methods when the initial vector is smooth. In particular, we derive a natural regularizing mechanism from first-principles.
\item We prove explicit computable error bounds that converge with rate $(m-1)$ in the number of quadrature nodes for initial vectors in $D(A^m)$ ($m\geq 2$).
\item We derive sharp asymptotics for the error bounds that yield explicit closed-form expressions for automatic near-optimal parameter selection in our scheme.
\end{itemize}
To achieve these results, we work in a ``discretization oblivious" paradigm, designing and analyzing the algorithm at the operator level before introducing and assessing the impact of discretizating the relevant infinite-dimensional operators~\cite{gilles2019continuous,horning2020feast,colbrook2021computing,horning2021computing,colbrook2022computing,colbrook2023computing,boulle2023learning,boulle2023elliptic}.
The main computational cost of our scheme is the numerical solution of shifted linear equations, $(z_k-A)u=x$, at the quadrature nodes $z_1,\ldots,z_N$. These can be solved with the user's choice of discretization or approximation technique. To make our scheme compatible with end-to-end error analysis for verified numerics, we also derive residual-based error bounds that are easy to compute for many standard discretizations.   

The paper is organized as follows. In~\cref{sec:background}, we review a few key properties of strongly continuous semigroups and their contour integral representations. In~\cref{sec:quad_schemes}, we introduce a family of high-order contour-integral methods accompanied by explicit error bounds and closed-form quasi-optimal parameter selection. \Cref{sec:practice} discusses practical algorithmic considerations related to inexact quadrature samples and approximation of infinite-dimensional operators. Proofs of the main results are given in~\cref{sec:error_bounds} and numerical experiments involving continuous-time Koopman semigroups are conducted in~\cref{sec:num_exp}. Technical details related to integration and quadrature in a Banach space are worked out in~\cref{app:quad_banach} and some useful asymptotic identities for hypergeometric functions are summarized in~\cref{app:hypogeo_func}.

\section{Contour integral representations of semigroups}\label{sec:background}

This section reviews some key facts from the theory of strongly continuous semigroups~\cite[Ch.~II]{engel2000one} and introduces a functional calculus that connects the operator exponential to the resolvent via regularized contour integrals~\cite{batty2013holomorphic}. The contour integral representations are the starting point for our numerical approximations and the semigroup theory provides tools to derive explicit bounds for the approximation error.

\subsection{Strongly continuous semigroups}\label{sec:sc_semigroups}

A strongly continuous semigroup on the Banach space $\mathcal{X}$ is a semigroup $(K(t))_{t\geq 0}$ of bounded linear operators on $\mathcal{X}$ such that each map $u_x:\mathbb{R}_+\rightarrow\mathcal{X}$ of the form $u_x(t)=K(t)x$ depends continuously on the parameter $t$. The generator of the semigroup is
$$
Ax = \lim_{t\rightarrow 0} \frac{1}{t}(K(t)x - x),\quad\text{with domain}\quad D(A)=\{x \in \mathcal{X}\,|\, \lim_{t\rightarrow 0}\smash{\frac{1}{t}} (K(t)x-x) \,\,{\rm exists}\}.
$$
When $D(A)=\mathcal{X}$, then $A$ is bounded and $(K(t))_{t\geq 0}$ is uniformly continuous on $\mathcal{X}$. In general, $A$ may be unbounded but it is always closed\footnote{A closed operator has a graph, the collection of pairs $(x,Ax)$ with $x\in D(A)$, that is closed in $\mathcal{X}\times \mathcal{X}$.} and densely defined in $\mathcal{X}$.

Conversely, the classic generation theorems of Hille, Yosida, Feller, Miyadera, and Phillips~\cite[Ch.~II, Thm.~3.5,~3.8]{engel2000one} characterize precisely which linear operators generate strongly continuous semigroups on $\mathcal{X}$ in terms of the resolvent operator, defined by
\begin{equation}
    R_A(z) = (z-A)^{-1}, \qquad\text{for all}\qquad z\in\rho(A).
\end{equation}
Here, $\rho(A)=\{z\in\mathbb{C} \,|\, (z-A):\mathcal{X}\rightarrow\mathcal{X}$ is bijective$\}$ is the resolvent set, an open set on which $R_A(z)$ is an analytic operator-valued function. The spectrum of $A$, where invertibility of $z-A$ fails, is the closed set $\lambda(A)=\mathbb{C}\setminus\rho(A)$. The norm on $\mathcal{X}$ is denoted by $\|\cdot\|$. For any strongly continuous semigroup, there are $M\geq 1$ and $\omega\in\mathbb{R}$ such that the semigroup satisfies a growth bound as in (a) below~\cite[Ch.~I, Prop.~5.5]{engel2000one}.
\begin{theorem}\label{thm:generation}
    Given a linear operator $A:D(A)\rightarrow\mathcal{X}$ on a Banach space $\mathcal{X}$ and constants $M\geq 1$ and $\omega\in\mathbb{R}$, the following statements are equivalent:
    \begin{itemize}
        \item[(a)] $A$ generates a strongly continuous semigroup $(K(t))_{t\geq 0}$ satisfying the growth bound
        $$
        \lVert K(t)\rVert \leq M\exp(\omega t)\qquad\text{for all}\qquad t\geq 0.
        $$
        \item[(b)] $A$ is closed, densely defined, and for every $z\in\mathbb{C}$ with ${\rm Re}\,z>\omega$, it holds that $z\in\rho(A)$ and
        $$
        \|R_A(z)^n\|\leq \frac{M}{({\rm Re}\,z - \omega)^n}, \qquad \text{for all}\qquad n\in\mathbb{N}_{\geq 0}.
        $$
    \end{itemize}
\end{theorem}

In the remainder of this paper, we assume that $A$ generates a strongly continuous semigroup as in (a) with $\omega=0$. Note that there is no loss of generality. If $\tilde A$ generates a $C^0$-semigroup with constants $M$ and $\omega$ in (a) of~\cref{thm:generation}, then $A = \tilde A - \omega$ generates a $C^0$-semigroup with constants $M$ and $0$. For our algorithmic framework, the key facts are the resolvent's analyticity and uniform boundedness in the half-plane ${\rm Re}\,z>\omega$.

\subsection{Functional calculus for semigroups}\label{sec:func_calc}

For a strongly continuous semigroup, the resolvent of the generator is the operator-valued Laplace transform of the operator exponential~\cite[Ch.~II, Theorem~1.10]{engel2000one},
\begin{equation}\label{eqn:lap_transform}
R_A(z)x = \int_{0}^\infty e^{-zt}\exp(At)x\,dt, \qquad \text{for all}  \qquad z\in\rho(A), \quad x\in\mathcal{X}.
\end{equation}
The inversion formula holds for all $x\in D(A)$. Given $\delta > 0$~\cite[Theorem~11.6.1]{hille1996functional},
\begin{equation}\label{eqn:Brom_inv}
\exp(At)x = \frac{1}{2\pi i}\int_{\delta-i\infty}^{\delta+i\infty} e^{zt}R_A(z)x\,dz, \qquad \text{for all} \qquad t>0, \quad x\in D(A).
\end{equation}
While the inversion formula provides a representation of the semigroup as a contour integral, analogous to~\cref{eqn:cont_idea}, the improper integral may not converge absolutely. Therefore, it is an unreliable starting point for a numerical approximation scheme. 

Instead, we follow Colbrook~\cite{colbrook2022computing} and use a functional calculus for half-plane operators to obtain an absolutely convergent analog of~\cref{eqn:Brom_inv}. Given a function $r(z)$ that is (i) analytic in an open half-plane containing the contour in~\cref{eqn:Brom_inv} and (ii) decays like $|r(z)|=\mathcal{O}(|z|^{-1-\sigma})$ for some $\sigma>0$, it holds that $K(t) = \exp(At)$ with~\cite{batty2013holomorphic}
\begin{equation}\label{eqn:func_calc}
    r(A)\exp(At)x = \frac{1}{2\pi i}\int_{\delta-i\infty}^{\delta+i\infty} r(z) e^{zt} R_A(z)x\,dz, \qquad \text{for all} \qquad t>0, \quad x\in \mathcal{X}.
\end{equation}
This improper integral converges absolutely and uniformly on $\mathcal{X}$. Consequently, we can approximate $\exp(At)x$ by approximating the integral and inverting $r(A)$. While Colbrook uses $r(z)=(\delta+1-z)^{-2}$ and focuses on computability and classification in the strongly continuous case~\cite{colbrook2022computing}, we use a family of higher-degree rational functions that exploit regularity in the vector $x$ for high-order accurate quadrature approximations.

\section{Quadrature approximations of semigroups}\label{sec:quad_schemes}

Our approximation scheme begins with the integral representation suggested in~\cref{eqn:func_calc},
\begin{equation}\label{eqn:reg_op_exp}
    \exp(At)x =r(A)^{-1}\left[\frac{1}{2\pi i}\int_{\delta-i\infty}^{\delta+i\infty} r(z) e^{zt} R_A(z)\,dz\right]x.
\end{equation}
Given an appropriate set of quadrature nodes, $z_{-N},\ldots,z_{N}$, along the contour with quadrature weights, $w_{-N},\ldots,w_N$, the corresponding quadrature approximation is
\begin{equation}\label{eqn:quad_approx}
\exp(At)x \approx r(A)^{-1}\left[\frac{1}{2\pi i}\sum_{k=-N}^N w_kr(z_k)e^{z_k t}R_A(z_k)\right]x.    
\end{equation}
To form the approximation in practice, one first samples the resolvent at each quadrature node by solving the linear systems $(z_k-A)u_k=x$ for $u_k\in D(A)$. To approximate $\exp(At)x$ at any time $t>0$, one computes the time-dependent weights, evaluates the linear combination of vectors $u_k$, $-N\leq k\leq N$, in brackets, and applies $r(A)^{-1}$.

The performance of this computational framework depends heavily on the choice of regularizing function $r(z)$ used in the half-plane functional calculus. Among other factors, it influences the cost of computing the operator-valued function $r(A)^{-1}$ and the accuracy of the quadrature rule. Therefore, this section focuses on the key question:

\begin{quote}
\begin{center}
    \textbf{How should we choose the regularizer $r(z)$?}
\end{center}
\end{quote}
After a simple illustration of the role of the regularizer in~\cref{sec:example_analysis}, we derive principled criteria for $r(z)$ in~\cref{sec:choose_regularizer} and propose a family of high-order schemes accompanied by explicit error bounds and automatic parameter selection in~\cref{sec:high_order_scheme}.

\subsection{Analysis of a simple regularized quadrature scheme}\label{sec:example_analysis}

\begin{figure}
    \centering
    \begin{minipage}{0.48\textwidth}
        \begin{overpic}[width=\textwidth]{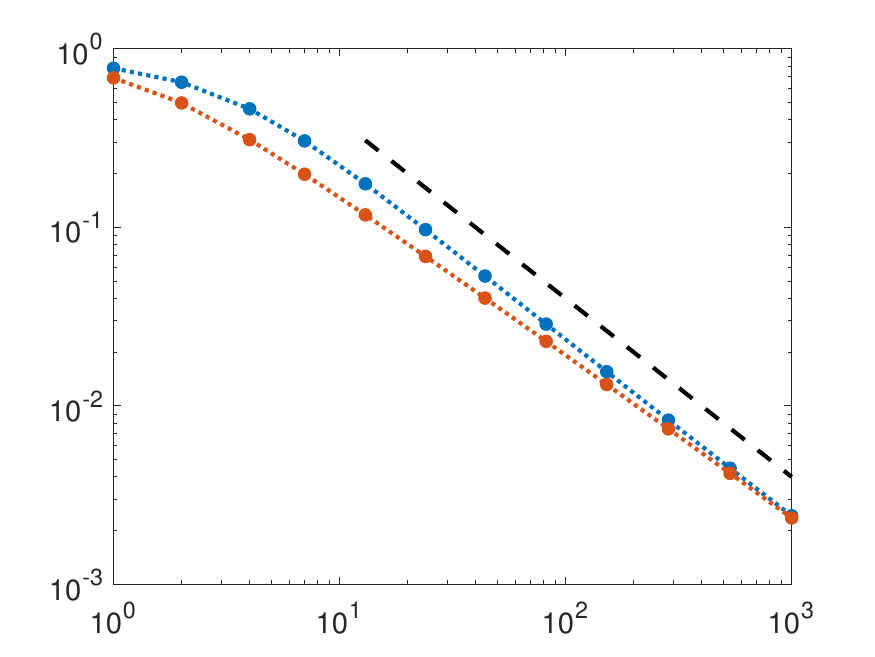}
        \put (48,1) {$\displaystyle N$}
        \put (20,72) {Error bounds (\Cref{thm:trap_rule_1_error})}
        \put (60,48) {\rotatebox{-40}{$\mathcal{O}(1/N)$}}
        \end{overpic}
    \end{minipage}
    \begin{minipage}{0.48\textwidth}
        \begin{overpic}[width=\textwidth]{figures/geometry_strongcont_sg.pdf}
        \put (48,1) {$\displaystyle  {\rm Re}\,z$}
        \put (2,36) {\rotatebox{90}{$\displaystyle {\rm Im}\,z$}}
        \put (68,42) {$\displaystyle \mathcal{O}(|z|^{-2})$}
        \put (75,50) {\tikz\draw[->,line width=1pt] (0,0) -- (0,0.75);}
        \put (75,24) {\tikz\draw[->,line width=1pt] (0,0) -- (0,-0.75);}
        \end{overpic}
    \end{minipage}
    \caption{\label{fig:2nd_order_bounds} In the left-hand panel, the total error bound (for $E_D+E_T$) from~\cref{thm:trap_rule_1_error} is plotted as a function of $N=$ `number of quadrature nodes' at $t=1$ with $\delta=a=2$ when the discretization parameter is fixed at $h=0.5$ (blue circles) and optimized numerically to minimize the total error bound (orange circles). The bounds are plotted with the normalization $\|(\delta+a-A)^2x\|=1$ and compared with the asymptotic convergence rate $\mathcal{O}(1/N)$ (black dashed line). The right panel illustrates the decay of the integrand along the contour due to the regularizer $r(z)=(\delta+a-z)^{-2}$.}
\end{figure}
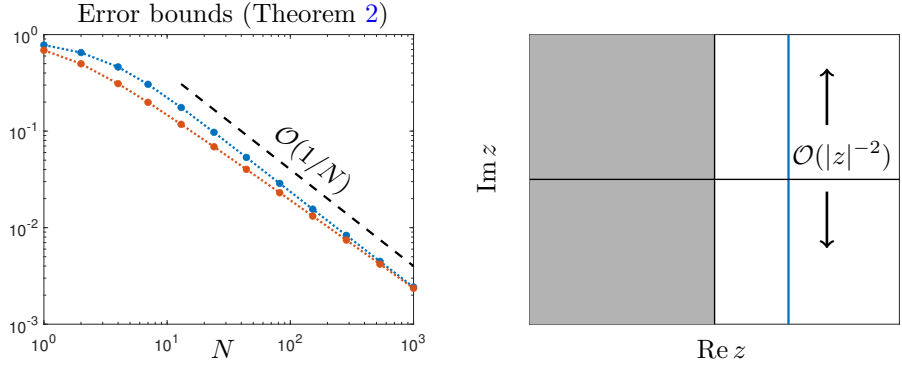

To develop criteria for $r(z)$, we illustrate a few key points by analyzing the approximation error in~\cref{eqn:quad_approx} for the simple choice $r(z)=(\delta+a-z)^{-2}$ with $a>0$. This choice is the simplest rational function that satisfies the analyticity and decay conditions required by the functional calculus in~\cref{sec:func_calc}. It also coincides with Colbrook's choice when $a=1$~\cite{colbrook2022computing}. The integral representation of the operator exponential is then
\begin{equation}\label{eqn:reg_op_exp2}
    \exp(At)x = (\delta+a-A)^2\left[\frac{1}{2\pi i}\int_{\delta-i\infty}^{\delta+i\infty} \frac{e^{zt} R_A(z)}{(\delta+a-z)^2}\,dz\right]x.
\end{equation}
Notice that the integrand is analytic and uniformly integrable along vertical contours within any strip of the form $\alpha\leq{\rm Re}\,z\leq\beta$ with $0<\alpha<\beta<\delta + a$.

To discretize this integral, we use the trapezoidal rule because it is accompanied by simple and sharp error bounds for integrands that are analytic and uniformly integrable in a complex strip~\cite{trefethen2014exponentially}.\footnote{While the trapezoidal rule is a natural choice due to the `strip' geometry of our analytic integrand, one may obtain similar results by truncating the contour and applying, e.g., a Gauss or Clenshaw-Curtis rule~\cite{trefethen2008gauss}.} Given $2N+1$ quadrature nodes with equal spacing $h$, i.e., $z_k = \delta + ihk$ for $-N\leq k\leq N$, the trapezoidal approximation to~\cref{eqn:reg_op_exp2} is
\begin{equation}\label{eqn:quad_approx_simple}
\exp(At)x \approx (\delta+a-A)^2\left[\frac{h}{2\pi}\sum_{k=-N}^N \frac{e^{(\delta+i h k) t}}{(a-i h k)^2}R_A(\delta + i h k)\right]x.    
\end{equation}
A simple extension of the trapezoidal rule's error analysis to the Banach space setting (see~\cref{app:quad_banach}) bounds the difference between bracketed terms in~\cref{eqn:reg_op_exp2,eqn:quad_approx_simple}. 

However, notice that $r(A)^{-1}=(\delta+a-A)^2$ is an unbounded operator on $\mathcal{X}$. Without knowledge of its regularity, the difference in the bracketed terms may be amplified without bound. To control this amplification, we require that $x\in D(r(A)^{-1}) = D(A^2)$. In this case, $r(A)^{-1}$ commutes with the bracketed terms in~\cref{eqn:reg_op_exp2,eqn:quad_approx_simple}. By incorporating $r(A)^{-1}$ into the integrand, its influence on the quadrature error is quantified by the size of the vector $r(A)^{-1}x$ in the error analysis for the trapezoidal rule.

\begin{theorem}\label{thm:trap_rule_1_error}
     Suppose that $A:D(A)\rightarrow\mathcal{X}$ satisfies condition (a) in~\cref{thm:generation} with constants $M$ and $\omega=0$. Given contour location $\delta>0$ and quadrature parameters $h>0$, $N\in\mathbb{N}_{\geq 1}$, then the approximation error $\|E_{h,N}(t;\delta)\|=E_{\rm D} + E_{\rm T}$ in~\cref{eqn:quad_approx_simple} at time $t>0$ satisfies the bound on the discretization error
     $$
     E_{\rm D} \leq \frac{M e^{\delta t}}{\delta a}\left[\frac{4e^{\sigma t/2}}{e^{\sigma\pi/h}-1}\right]\lVert(\delta+a-A)^2x\rVert,
     $$
     and the bound on the truncation error
     $$
     E_{\rm T} \leq \frac{M e^{\delta t}}{\delta a}\left[\frac{1}{2}-\frac{1}{\pi}\arctan\left(\frac{hN}{a}\right)\right]\lVert(\delta+a-A)^2x\rVert
     $$
     where $\sigma = \min(\delta,a)$, provided that $x$ satisfies the regularity condition $x\in D(A^2)$.
\end{theorem}
\begin{proof}
    See~\cref{sec:error_bounds}.
\end{proof}

The error bound in~\cref{thm:trap_rule_1_error} is a sum of two terms. The first term corresponds to the discretization error, the error made when replacing the integral in~\cref{eqn:reg_op_exp2} by an infinite series of samples at equally spaced quadrature nodes along the contour. The second term corresponds to the truncation of this infinite series to a centered combination of just $2N+1$ samples. The discretization error depends only on the node spacing and vanishes exponentially fast as the node spacing $h\rightarrow 0$. However, the truncation error depends on the decay of the integrand at the equally spaced nodes along the contour. Due to the slow decay of the integrand in~\cref{eqn:reg_op_exp2}, the truncation error bound in~\cref{thm:trap_rule_1_error} decreases at a rate of $\mathcal{O}(1/N)$ as $N\rightarrow\infty$ (see~\cref{fig:2nd_order_bounds}).

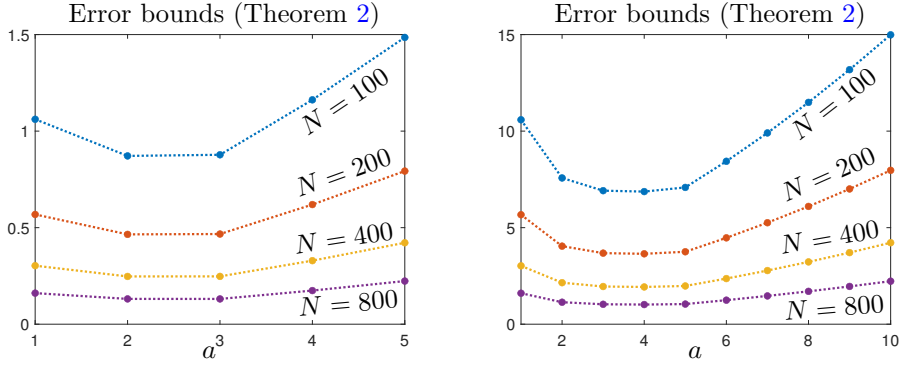
\begin{figure}
    \centering
    \begin{minipage}{0.48\textwidth}
        \begin{overpic}[width=\textwidth]{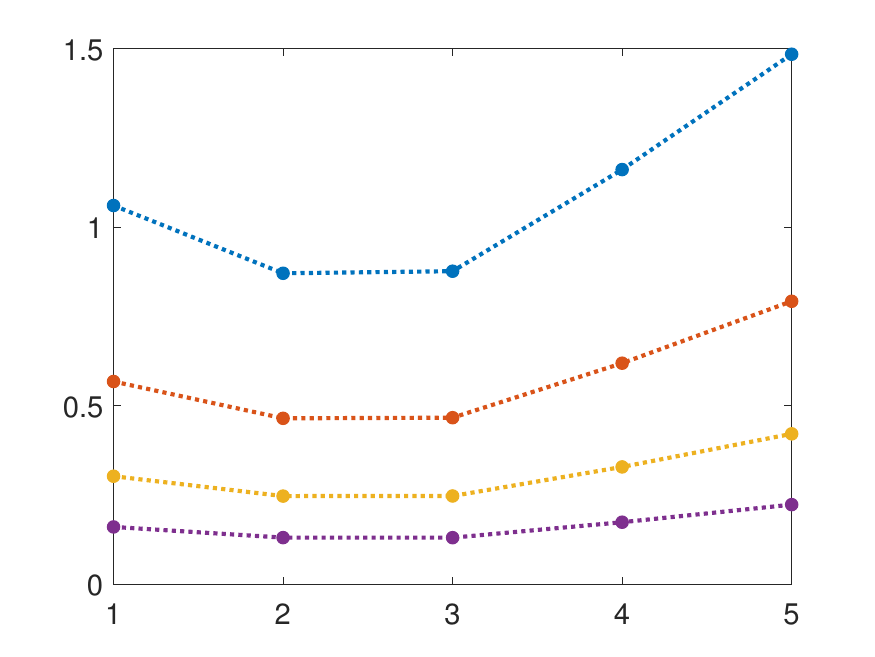}
        \put (48,1) {$\displaystyle a$}
        \put (20,72) {Error bounds (\Cref{thm:trap_rule_1_error})}
        \put (68,48) {\rotatebox{34}{$\displaystyle N=100$}}
        \put (67,35) {\rotatebox{20}{$\displaystyle N=200$}}
        \put (67,23) {\rotatebox{9}{$\displaystyle N=400$}}
        \put (68,10) {\rotatebox{5}{$\displaystyle N=800$}}
        \end{overpic}
    \end{minipage}
    \begin{minipage}{0.48\textwidth}
        \begin{overpic}[width=\textwidth]{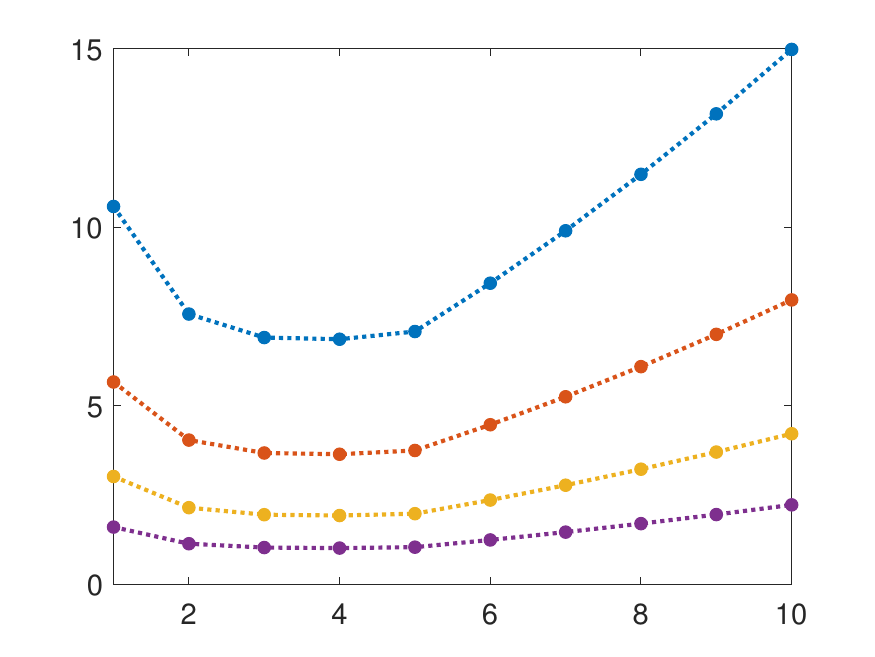}
        \put (48,1) {$\displaystyle a$}
        \put (20,72) {Error bounds (\Cref{thm:trap_rule_1_error})}
        \put (69,47) {\rotatebox{38}{$\displaystyle N=100$}}
        \put (67,34) {\rotatebox{24}{$\displaystyle N=200$}}
        \put (67,23) {\rotatebox{10}{$\displaystyle N=400$}}
        \put (68,9) {\rotatebox{5}{$\displaystyle N=800$}}
        \end{overpic}
    \end{minipage}
    \caption{\label{fig:pole_opt} The total error bound (for $E_D+E_T$) from~\cref{thm:trap_rule_1_error} is plotted as a function of $a=$ `pole location' of the regularizar $r(z)=(\delta+a-z)^{-2}$ for contour locations $\delta = 3$ (left panel) and $\delta=5$ (left panel) and $N=100$ (blue), $N=200$ (red), $N=400$ (yellow), and $N=800$ (purple). In both panels, the error bound was computed at $t=1$ and normalized so that $\|(\delta+a-A)^2x\|=(\delta+a)^2$. The discretization parameter $h$ was optimized numerically to minimize the total error bound.}
\end{figure}

\subsection{Criteria for the regularizing function}\label{sec:choose_regularizer}

Having analyzed the basic performance of the quadrature approximation for a simple regularizer in~\cref{sec:example_analysis}, we now identify three key factors to constrain and guide our choice of $r(z)$. Our goal is to choose a regularizer that, if possible, leverages the regularity of $x$ to accelerate the convergence of the quadrature approximation.

\textbf{Factor 1: Complexity of $r(A)^{-1}$.} To avoid making the problem more difficult than the original goal of computing $\exp(At)$, it is reasonable to require that $r(z)$ be a rational function, i.e., the ratio of two polynomials. In this case, $r(A)^{-1}$ is also rational and can be computed by applying powers of $A$ and its resolvent to vectors in $\mathcal{X}$. To keep the computational cost low, we would like the degrees of the numerator, $p$, and denominator, $q$, to be as low as possible. Therefore, we assume (without loss of generality) that the numerator and denominator have no common factors. Moreover, the decay requirement for $r(z)$ in~\cref{eqn:func_calc} indicates that $q\geq 2$ and $0\leq p\leq q-2$. 

\textbf{Factor 2: Regularity of $x$.} The operator $r(A)^{-1}$ will typically be an unbounded operator on $\mathcal{X}$ due to the decay condition in~\cref{eqn:func_calc}, as we encountered in~\cref{sec:choose_regularizer}. To develop rigorous error bounds for our approximation scheme, we mitigate the power of $r(A)^{-1}$ to amplify errors in the quadrature approximation when passing from~\cref{eqn:reg_op_exp} to~\cref{eqn:quad_approx} by leveraging the regularity of $x$. If $x\in D(r(A)^{-1})=D(A^{q-p})$, the order of operations in~\cref{eqn:reg_op_exp} can be changed so that 
\begin{equation}\label{eqn:reg_op_exp3}
    \exp(At)x =\frac{1}{2\pi i}\int_{\delta-i\infty}^{\delta+i\infty} r(z) e^{zt} R_A(z)\left[r(A)^{-1}x\right]\,dz.
\end{equation}
By absorbing $r(A)^{-1}$ into the integrand, its impact on the quadrature error can be rigorously accounted for in error bounds, as demonstrated in~\cref{thm:trap_rule_1_error}. To enable the interchange in~\cref{eqn:reg_op_exp3}, we limit the degree of the denominator in $r(z)$ based on the regularity of $x$. So, if $x\in D(A^m)$ for some integer power $m$, we require that $q-p \leq m$.

\textbf{Factor 3: Analyticity and decay of $r(z)$.} \Cref{thm:trap_rule_1_error} illustrates how the quadrature error in the trapezoidal rule typically improves as two features of the integrand increase: the width of its strip of analyticity around the contour and the rate of its decay along the contour. In light of the conditions from Factors 1 and 2, the maximum rate of decay for $r(z)$ is achieved with $q-p=m$. Noting that the integrand is analytic in the strip $0<{\rm Re}\,z \leq \delta$ to the left of the contour, it is reasonable to require a comparable strip of analyticity to the right of the contour. Therefore, we place the poles of $r(z)$ to the right of the vertical line ${\rm Re}\, z = 2\delta$. Note that this restriction avoids placing artificial restrictions on the strip of analyticity, allowing us to take larger values of the spacing, $h$, when balancing truncation and discretization errors. This amplifies the impact of each additional quadrature point on the truncation error as $N$ increases.

Considering Factors~1-3, we arrive at a rational regularizer of type $(p,q)$ with degree condition $q-p=m$ whose poles are to the right of the vertical line ${\rm Re}\,z=2\delta$. Since the choice of numerator polynomial is not explicitly constrained by Factors~2 or 3, we defer to Factor~1 and minimize the total degree of the regularizer by choosing $p=0$. The result is a rational function with poles $s_1,\ldots,s_m$ leading to
\begin{equation}\label{eqn:rational_regularizer}
    r(z) = (s_1-z)^{-1}\ldots(s_m-z)^{-1}, \qquad\text{where}\qquad {\rm Re}\,s_k\geq 2\delta, \quad k=1,\ldots,m.
\end{equation}
The corresponding operator-valued function to apply is $r(A)^{-1}=(s_1-A)\ldots(s_m-A)$.

To select the poles in~\cref{eqn:rational_regularizer}, consider the following heuristic `greedy' optimization procedure. Fixing all but one pole, where should this last pole be moved to reduce approximation error? To maximize the decay of the function in both directions along the vertical contour, the pole should clearly lie on the real axis. At first, one might be tempted to place the pole very far from the contour on the real axis, however, this comes with a price. As the pole is moved to the right along the real axis, the norm $\|r(A)^{-1}x\|$ increases proportionally to the decrease in the tails of $r(z)$. In other words, scaling $r(z)$ has no net effect because it is immediately offset by the presence of $r(A)^{-1}$. To account for this explicitly, suppose that $r(z)$ is normalized so that $r(\delta) = 1$. Then, the fastest decay is achieved when the pole is as close as possible to the contour (taking into account our earlier constraint), i.e., when the pole is placed at $z=2\delta$. Applying the argument to each pole, we obtain the poles $s_k = 2\delta$, for $k=1,\ldots,m$, in~\cref{eqn:rational_regularizer}.

\subsection{High-order approximation schemes}\label{sec:high_order_scheme}

With the new $m$th-order rational regularizer $r(z)=(2\delta-z)^{-m}$ in hand, we propose a family of high-order approximation schemes based on a trapezoidal approximation of~\cref{eqn:reg_op_exp}. Given an even integer order $m\geq 2$, the $m$th-order approximation is\footnote{In principle and practice, there is no problem using odd integer orders $m\geq 2$. However, some manipulations with hypergeometric functions in the error analysis simplify when $m/2$ is an integer.}
\begin{equation}\label{eqn:quad_approx_mth-order}
\exp(At)x \approx (2\delta-A)^m\left[\frac{h}{2\pi}\sum_{k=-N}^N \frac{e^{(\delta+i h k) t}}{(\delta-i h k)^m}R_A(\delta + i h k)\right]x.    
\end{equation}
When $x\in D(A^m)$, we provide explicit bounds on the approximation error in~\cref{eqn:quad_approx_mth-order}. Here, $\Gamma(x)$ is the gamma function and ${}_2 F_1(a,b;c;z)$ is the hypergeometric function~\cite{NIST:DLMF}.

\begin{theorem}\label{thm:trap_rule_2_error}
     Suppose that $A:D(A)\rightarrow\mathcal{X}$ satisfies condition (a) in~\cref{thm:generation} with constants $M$ and $\omega=0$. Given contour location $\delta>0$ and quadrature parameters $h>0$, $N\in\mathbb{N}_{\geq 1}$, then the approximation error $\|E_{h,N}(t;\delta)\|=E_{\rm D}+E_{\rm T}$ in~\cref{eqn:quad_approx_mth-order} at time $t>0$ satisfies the bound on the discretization error,
     $$
     E_{\rm D} \leq \frac{M e^{3\delta t/2}}{\delta^m(e^{\delta\pi/h}-1)} \left[\frac{2^{m+1}\Gamma\left(\frac{3}{2}\right)\Gamma\left(\frac{m-1}{2}\right)}{\pi\Gamma\left(\frac{m}{2}\right)}\right]\lVert(2\delta-A)^m x\rVert,
     $$
     and the bound on the truncation error,
     $$
     E_{\rm T} \leq  \frac{M e^{\delta t}}{\delta^m}\left[\frac{\Gamma\left(\frac{3}{2}\right)\Gamma\left(\frac{m-1}{2}\right)}{\pi\Gamma\left(\frac{m}{2}\right)}-\frac{hN}{\pi\delta} {}_2F_1\left(\frac{1}{2},\frac{m}{2};\frac{3}{2};-\left(\frac{hN}{\delta}\right)^2\right)\right]\lVert(2\delta-A)^m x\rVert,
     $$
     provided that $x$ satisfies the regularity condition $x\in D(A^m)$.
\end{theorem}
\begin{proof}
    See~\cref{sec:error_bounds}.
\end{proof}

As in~\cref{thm:trap_rule_1_error}, the discretization error decreases exponentially with rate $\delta$ (the width of the integrand's strip of analyticity). The truncation error decreases at the algebraic rate of $\mathcal{O}((1/N)^{m-1})$ as $N\rightarrow\infty$, although this is obscured by the hypergeometric function. To see the scaling clearly, we expand the term in brackets to leading order in $hN/\delta$ (see~\cref{app:hypogeo_func}). The upper bound on the truncation error satisfies
\begin{equation}\label{eqn:asymptotic_truncation_bound}
    \left[\text{bound on}\,\,E_{\rm T}\right]\sim \frac{Me^{\delta t}}{\pi \delta^m}\left[\frac{1}{m-1}\left(\frac{\delta}{hN}\right)^{m-1}\right]\|(2\delta-A)^mx\|, \qquad\text{as}\qquad N\rightarrow\infty.
\end{equation}
In other words, the $m$th order scheme exploits regularity in $x$ to achieve convergence at a rate of $\mathcal{O}(1/N^{m-1})$ as $N\rightarrow\infty$. We also note that when $a=\delta$ in~\cref{thm:trap_rule_1_error} and $m=2$ in~\cref{thm:trap_rule_2_error}, the error bounds in the two theorems coincide exactly (see~\cref{app:hypogeo_func} for simplification of the gamma and hypergeometric functions).

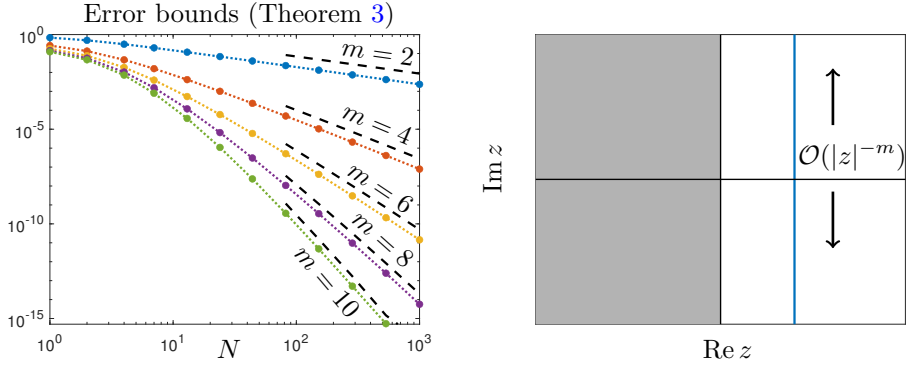
\begin{figure}
    \centering
    \begin{minipage}{0.48\textwidth}
        \begin{overpic}[width=\textwidth]{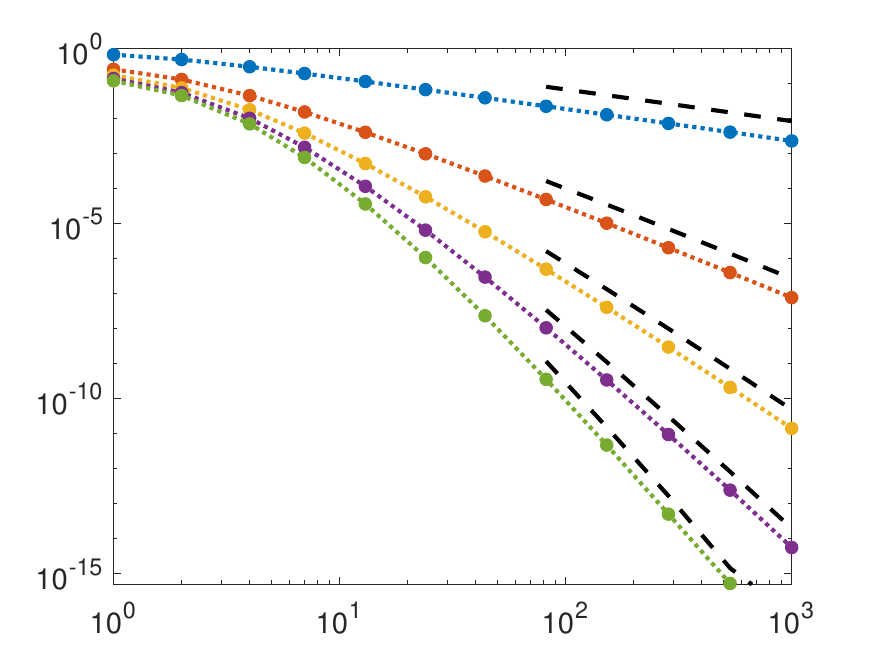}
        \put (48,1) {$\displaystyle N$}
        \put (20,72) {Error bounds (\Cref{thm:trap_rule_2_error})}
        \put (74,65) {\rotatebox{-8}{$m=2$}}
        \put (74,52) {\rotatebox{-23}{$m=4$}}
        \put (75,40) {\rotatebox{-34}{$m=6$}}
        \put (76,30) {\rotatebox{-42}{$m=8$}}
        \put (64,23) {\rotatebox{-50}{$m=10$}}
        \end{overpic}
    \end{minipage}
    \begin{minipage}{0.48\textwidth}
        \begin{overpic}[width=\textwidth]{figures/geometry_strongcont_sg.pdf}
        \put (48,1) {$\displaystyle  {\rm Re}\,z$}
        \put (2,36) {\rotatebox{90}{$\displaystyle {\rm Im}\,z$}}
        \put (68,42) {$\displaystyle \mathcal{O}(|z|^{-m})$}
        \put (75,50) {\tikz\draw[->,line width=1pt] (0,0) -- (0,0.75);}
        \put (75,24) {\tikz\draw[->,line width=1pt] (0,0) -- (0,-0.75);}
        \end{overpic}
    \end{minipage}
    \caption{\label{fig:mth_order_bounds} In the left-hand panel, the total error bounds (for $E_D+E_T$) from~\cref{thm:trap_rule_1_error} are plotted as a function of $N=$ `number of quadrature nodes' at $t=1$ with $\delta=2$ for $m=2,4,6,8$. The discretization parameter is optimized numerically to minimize the total error bound. The bounds are plotted with the normalizations $\|(2\delta-A)^mx\|=2^{m-2}$ to reflect typical graph norm growth of smooth functions, and compared with the asymptotic convergence rates $\mathcal{O}(1/N^{m-1})$ (black dashed lines). The right-hand panel shows the integrand's decay on the contour due to the regularizer $r(z)=(2\delta-z)^{-m}$. }
\end{figure}

The closed-form expressions for the error bounds in~\cref{thm:trap_rule_2_error} also allow us to select asymptotically optimal quadrature parameters to meet a uniform target accuracy $\epsilon>0$ over a time horizon $[0,T]$. First, requiring that the discretization error is less than $\epsilon/2$ at time $T$ and solving for the quadrature spacing, $h$, yields
\begin{equation}\label{eqn:choose_h}
    h = \pi\delta\left[\log\left(1 + \frac{1}{\epsilon}\frac{e^{3\delta t/ 2}}{\delta^m}\left(\frac{2^{m+1}\Gamma\left(\frac{3}{2}\right)\Gamma\left(\frac{m-1}{2}\right)}{\pi\Gamma\left(\frac{m}{2}\right)}\right)\|(2\delta-A)^mx\|\right)\right]^{-1}.
\end{equation}
Second, requiring that the bound for the truncation error also be on the order of $\epsilon/2$ and using the leading-order asymptotic in~\cref{eqn:asymptotic_truncation_bound} to solve for $N$, we obtain that
\begin{equation}\label{eqn:choose_n}
    N = \left\lceil\frac{1}{h}\left[\frac{1}{\epsilon}\frac{e^{\delta t}}{\pi\delta}\frac{\|(2\delta-A)^mx\|}{m-1}\right]^{1/(m-1)}\right\rceil.
\end{equation}
Since the discretization and truncation error are both montonically increasing on $[0,T]$, the error bounds at time $T$ are also valid for the entire time-window ending at time $T$.

\begin{figure}
    \centering
    \begin{minipage}{0.48\textwidth}
        \begin{overpic}[width=\textwidth]{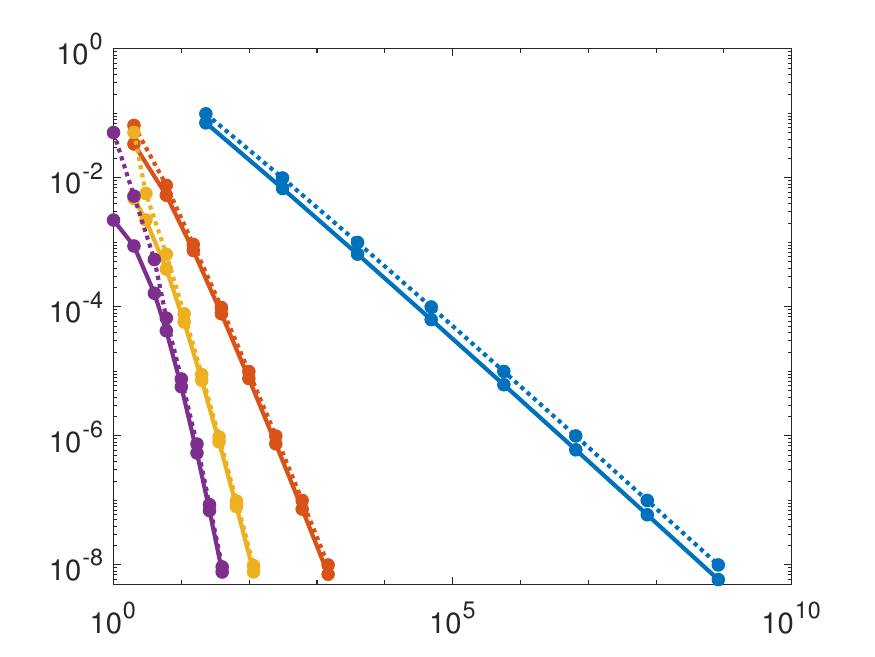}
        \put (42,1) {$\displaystyle N$}
        \put (20,72) {Error bounds (\Cref{thm:trap_rule_2_error})}
        \put (72,22) {\rotatebox{-42}{$m=2$}}
        \put (34,24) {\rotatebox{-66}{$m=4$}}
        \put (27,24) {\rotatebox{-74}{$m=6$}}
        \put (17,24) {\rotatebox{-80}{$m=8$}}
        \end{overpic}
    \end{minipage}
    \begin{minipage}{0.48\textwidth}
        \begin{overpic}[width=\textwidth]{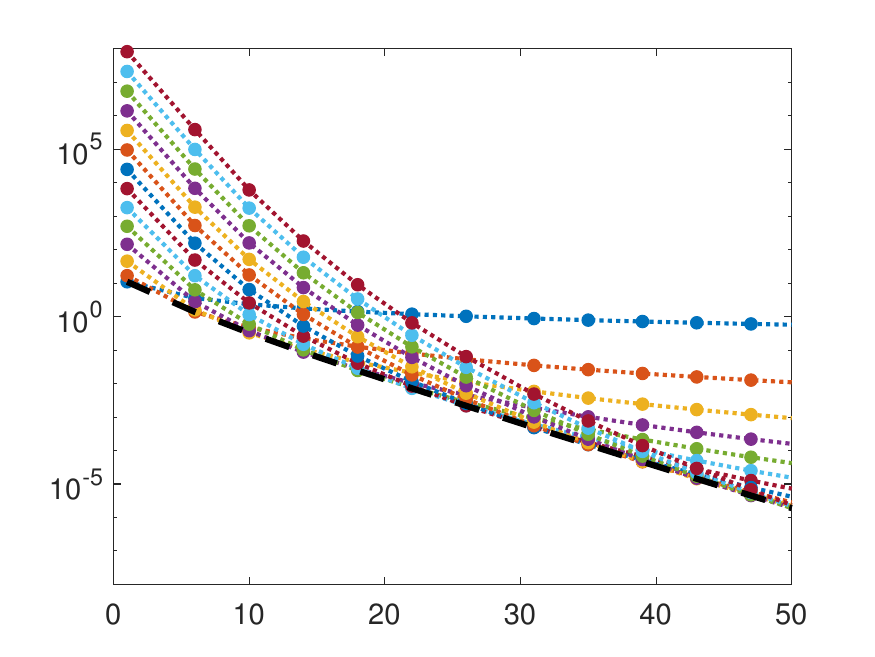}
        \put (48,1) {$\displaystyle N$}
        \put (20,72) {Error bounds (\Cref{thm:trap_rule_2_error})}
        \put (48,22) {\rotatebox{-16}{$\mathcal{O}(\exp(\alpha N))$}}
        \end{overpic}
    \end{minipage}
    \caption{\label{fig:parameter_envelope} The left panel compares the total error bounds (for $E_D+E_T$) from~\cref{thm:trap_rule_1_error}, plotted as a function of $N=$ `number of quadrature nodes' at $t=1$ with $\sigma=\delta=2$ for $m=2,4,6,8$. The discretization parameters $N$ and $h$ are chosen according to~\cref{eqn:choose_h,eqn:choose_n} for target error tolerances of $\epsilon=10^{-1},10^{-2},\ldots,10^{-8}$ (dotted lines) and compared with the error bound for optimal $h$ at the same values of $N$ (solid lines). For $N\geq 10$, the two are nearly indistinguishable, indicating that~\cref{eqn:choose_h,eqn:choose_n} are  nearly optimal. All bounds are normalized with $\|(2\delta-A)^mx\|=2^{m-2}$. The right panel demonstrates exponential convergence for functions in $D(A^\infty)$ with norm $\|(2\delta-A)^mx\|\leq (2\delta)^{m-2}$, reflected in the linear envelope (black dashes) of the $m$th order schemes on the semilog plot.}
\end{figure}
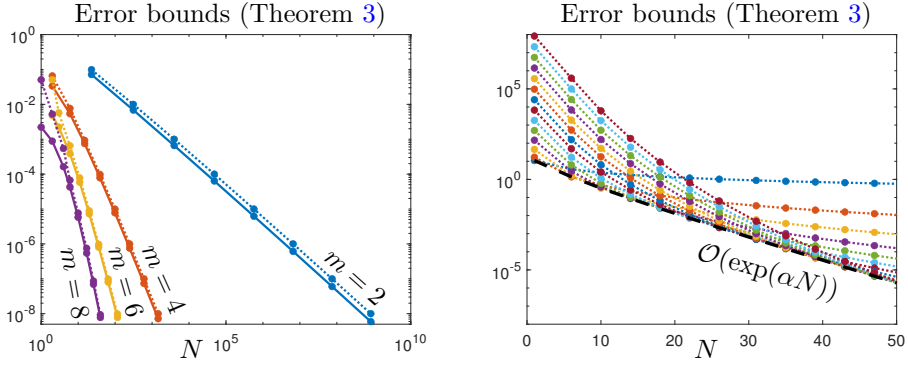

\section{Practical considerations for inexact samples}\label{sec:practice}

In practice, the quadrature error is not the only source of error in the scheme because the resolvent samples $R_A(z_k)$, for $k=-N,\ldots,N$, must be approximated. Moreover, the cost of approximating $R_A(z)$ to a given accuracy may vary widely with the location of $z$ in the right half-plane. This section provides a few practical tools to handle these approximation errors and understand their influence on the algorithm's performance.

Our first tool provides a computable \textit{a posteriori} error bound for the resolvent samples. The key ingredients are a computed residual together with the resolvent bound for strongly continuous semigroups in~\cref{thm:generation} (b). Suppose that $\tilde u_k\in D(A)$ is an approximate solution to the equation $(z_k-A)u_k=x$ and denote the residual by
$$
r_k = (z_k-A)\tilde u_k - x.
$$
The error between the computed solution $\tilde u_k$ and the true solution $u_k$ satisfies
$$
\tilde u_k - u_k = (z_k-A)^{-1}(r_k+x) - u_k = (z_k-A)^{-1}r_k.
$$
Since the quadrature nodes lie on the line ${\rm Re}\,z=\delta>0$ and item (b) of~\cref{thm:generation} limits the amplification power of the resolvent, we conclude that the error satisfies
\begin{equation}\label{eqn:residual_bound}
\|\tilde u_k - u_k\| \leq \frac{\|r_k\|}{\delta}, \qquad \text{for} \qquad k=-N,\ldots,N.
\end{equation}
As long as one can compute (or estimate) the norm of the residuals at each quadrature node, the resolvent bound immediately converts the residual norm to an error bound (or estimate) for the computed solution. For many practical problems the action of $A$ is known exactly on a finite-dimensional subspace of $D(A)$ containing the computed solution $\tilde u_k$ and the residual $r_k$ can be computed exactly or estimated reliably.

The error bound in~\cref{eqn:residual_bound} indicates that the actual error may be much larger than the residual when the contour is too close to the imaginary axis. This is a natural reflection of the possibility of ill-conditioning near the spectrum of $A$ in the left half-plane. However, \cref{eqn:residual_bound} guarantees that this effect is mild unless the contour is placed within a few decimal places from the imaginary axis. On the other hand, we often observe that the cost of achieving a small residual often increases rapidly as the contour is moved toward the imaginary axis (see the left panel of~\cref{fig:contour_choice2}). We offer the following interpretation in light of~\cref{eqn:lap_transform} and the final value theorem for Laplace transforms,
$$
\lim_{t\uparrow \infty} K(t)x = \lim_{s\downarrow 0} sR_A(s)x, \qquad x\in\mathcal{X}.
$$
In other words, the behavior of $R_A(z)$ at the origin captures the large-time asymptotic behavior of the semigroup $K(t)x$. While $K(t)x$ maintains the smoothness of the initial vector, e.g., $x\in D(A^m)$ implies $K(t)x\in D(A^m)$, many interesting classes of strongly continuous semigroups approach vectors in $\mathcal{X}$ with singular features (not in $D(A^m)$) as $t\rightarrow\infty$ (c.f. Example~3 in~\cref{sec:num_exp}). When the contour is near the origin, these singular features may appear in the resolvent samples and often require more computational degrees of freedom to resolve accurately (see the right panel of~\cref{fig:contour_choice2}).

Although placing the contour close to the imaginary axis can incur a computational cost in the resolvent samples, there are also factors that prevent us from placing the contour too far from the imaginary axis. For example, the error bounds in~\cref{thm:trap_rule_2_error} grow exponentially with $\delta$. In some numerical experiments, we have also observed the sum in the quadrature approximation of~\cref{eqn:quad_approx_mth-order} become ill-conditioned at large $\delta$.

\begin{figure}
    \centering
    \begin{minipage}{0.48\textwidth}
        \begin{overpic}[width=\textwidth]{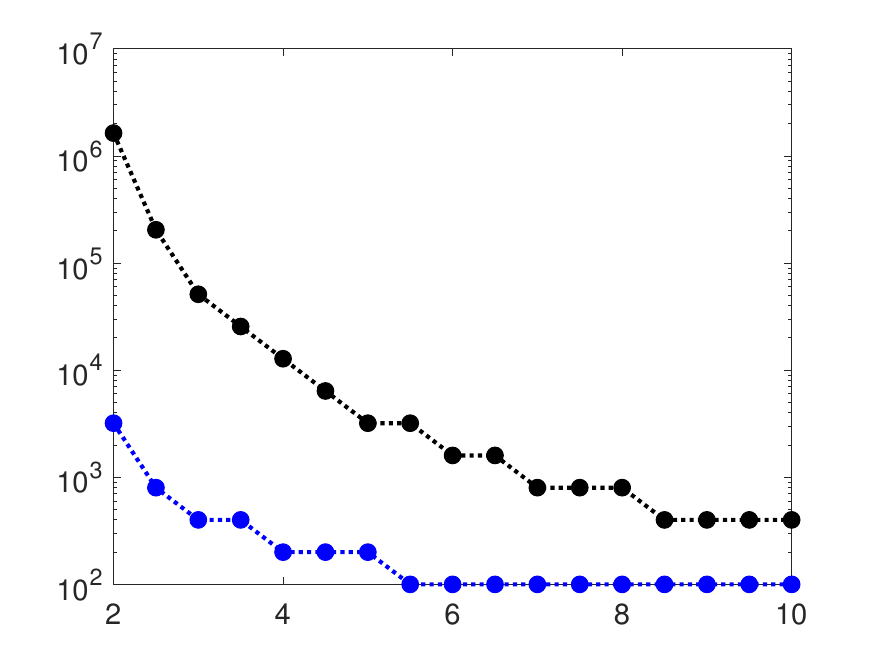}
        \put (48,-1) {$\displaystyle \delta$}
        \put (46,73) {$\displaystyle n(\epsilon')$}
        \end{overpic}
    \end{minipage}
    \begin{minipage}{0.48\textwidth}
        \begin{overpic}[width=\textwidth]{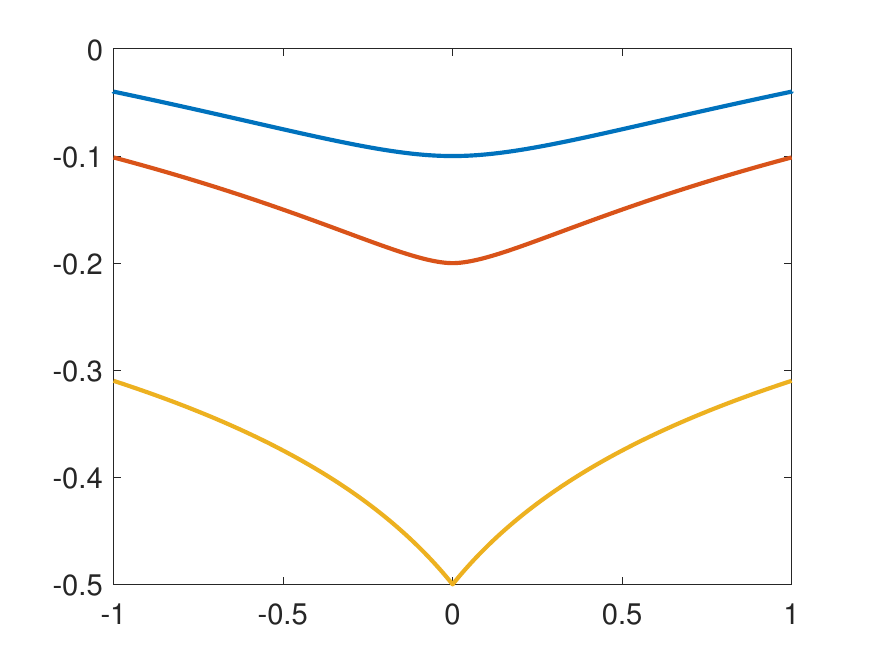}
        \put (48,-1) {$\displaystyle y$}
        \put (40,72) {$\displaystyle [R_A(z)x](y)$}
        \end{overpic}
    \end{minipage}
    \caption{\label{fig:contour_choice2} The discretization size, $n(\epsilon')$, needed to approximate $R_A(z)x$, for Example~3 in~\cref{sec:num_exp}, to within relative error $\epsilon' = 10^{-4}$ (blue) and $\epsilon'=10^{-8}$ (black) is plotted against the contour location $\delta>0$ in the left panel. In the right panel, the resolvent requires higher degree polynomial approximations (i.e.,  larger spectral discretizations) with the decreasing values of $\delta=10$ (blue), $5$ (red), and $2$ (yellow) because $R_A(z)x$ develops a cusp at the point $y=0$.}
\end{figure}

\section{Derivation of explicit error bounds}\label{sec:error_bounds}

Let $I(t)$ denote the exact value of the contour integral in~\eqref{eqn:reg_op_exp} after parametrization,
\begin{equation}\label{eqn:param_integral}
I(t) = \left[\frac{1}{2\pi}\int_{-\infty}^{\infty} r(\delta+is) e^{t(\delta+is)} R_A(\delta+is)\,ds\right] r(A)^{-1}x, \qquad t\geq 0. 
\end{equation}
We derive bounds on the approximation error in~\cref{eqn:quad_approx} in two steps. In the first step, we bound the discretization error caused by approximating the integral in~\cref{eqn:param_integral} by
\begin{equation}\label{eqn:inf_sum}
    S_h(t;\delta) = \left[\frac{h}{2\pi}\sum_{k=-\infty}^\infty r(\delta+ihk)e^{t(\delta+ihk)}R_A(\delta+ihk)\right]r(A)^{-1}x.
\end{equation}
In the second step, we bound the truncation error, that is, the error made by truncating the series in~\cref{eqn:inf_sum} to the finite sum on the right-hand side of~\cref{eqn:quad_approx}, which we denote here by $S_h^N(t;\delta)$. The total error in the approximation of $I(t)$ is bounded by
\begin{equation}\label{eqn:approx_error}
|I(t)-S_h^N(t;\delta)|\leq |I(t)-S_h(t;\delta)| + |S_h(t;\delta)-S_h^N(t;\delta)|.
\end{equation}
The first left-hand term is the \textit{discretization error}. The second is the \textit{truncation error}. After computing bounds for both error contributions for a general $r(z)$, we work out the key terms in detail to derive the explicit, interpretable bounds in~\cref{thm:trap_rule_1_error,thm:trap_rule_2_error}.

\subsection{Discretization error}\label{sec:disc_error}

The discretization error can be bounded above by applying a simple Banach space extension (see~\cref{app:quad_banach}) of error bounds for the trapezoidal rule~\cite[Thm.~5.1]{trefethen2014exponentially}.
\begin{theorem}\label{thm:trap_rule_error}
    Suppose $f:\mathbb{C}\rightarrow\mathcal{X}$ is analytic in the strip $|{\rm Im}\,z|<a$ for some $a>0$, that $f(z)\rightarrow 0$ uniformly as $|z|\rightarrow\infty$ in the strip, and for some $C_f>0$, it satisfies
    \begin{equation}\label{eqn:uniform_int_bound}
        \int_{-\infty}^\infty \|f(s+ib)\|\,ds\leq C_f, \qquad\text{for all}\qquad -a < b < a.
    \end{equation}
    If, for some $h>0$, the sum $h\sum_{k=-\infty}^\infty f(hk)$ converges absolutely, then it satisfies
    \begin{equation}\label{eqn:trap_rule_bound}
        \left\|\int_{\infty}^\infty f(s)\,ds - h\sum_{k=-\infty}^\infty f(hk)\right\|\leq \frac{2C_f}{e^{2\pi a/h}-1}.
    \end{equation}
    Moreover, the constant $2C_f$ in the numerator is sharp, i.e., as small as possible.
\end{theorem}

The integrand in~\eqref{eqn:param_integral} is analytic and uniformly integrable in the strip ${\rm |Im}\,z|<\delta/2$ because of the resolvent bound in~\cref{thm:generation} (b) and the decay condition on $r(z)$. Furthermore, we can calculate a suitable bound $C_f$ for~\cref{thm:trap_rule_error} explicitly.

\begin{lemma}\label{thm:uniform_int_bound}
     Suppose that $A:D(A)\rightarrow\mathcal{X}$ satisfies condition (a) in~\cref{thm:generation} with constants $M$ and $\omega=0$, that $x\in D(r(A)^{-1})$. Given contour location $\delta>0$ in~\cref{eqn:param_integral}, suppose also that $r(z)$ is analytic in the left half-plane $|{\rm Im}\,z|<2\delta$. Then $I(t)$ satisfies the uniform absolute integrability estimate in~\cref{eqn:uniform_int_bound} with $a=\delta/2$ and
     $$
     C_f = \frac{2M}{\delta}\frac{e^{3\delta t/2}}{2\pi}\max_{|b-\delta|<\delta/2}\left[\int_{-\infty}^\infty |r(\delta-b+is)|\,ds\right]\|r(A)^{-1}x\|.
     $$
\end{lemma}
\begin{proof}
     Denote the integrand of~\cref{eqn:param_integral} by $f(s;\delta,t)=r(\delta+is) e^{t(\delta+is)} R_A(\delta+is) r(A)^{-1}x$. After shifting the contour within the strip $\delta-\delta/2 < {\rm Re}\,z <\delta+\delta/2$, we take the norm of the integrand in~\cref{eqn:param_integral} and apply the resolvent bound in~\cref{thm:generation} (b) to obtain
    \begin{align*}
    \frac{1}{2\pi}\int_{-\infty}^{\infty}\| f(s+ib)\|\,ds 
    \leq \frac{M}{(\delta-b)}\frac{e^{t(\delta-b)}}{2\pi}\left[\int_{-\infty}^\infty |r(\delta-b+is)|\,ds\right]\|r(A)^{-1}x\|
    \end{align*}
    Taking the maximum of each factor over the strip $\delta/2<b<3\delta/2$ establishes $C_f$.
\end{proof}

The bounds on the discretization error in~\cref{thm:trap_rule_1_error,thm:trap_rule_2_error} follow directly from~\cref{thm:trap_rule_error} and~\cref{thm:uniform_int_bound} after bounding the max modulus integral in brackets in~\cref{thm:uniform_int_bound} for the regularizers $r(z)=(\delta+a-z)^{-2}$ and $r(z)=(2\delta-z)^{-m}$, respectively.
\begin{proof}[Proof of discretization error in~\cref{thm:trap_rule_2_error}]
    With $r(\delta-b+is) = (\delta-b+is)^{-m}$, so that $|r(\delta-b+is)|=((\delta-b)^2+s^2)^{-m/2}$, we first compute the indefinite integral associated with~\cref{thm:uniform_int_bound} using a hypergeometric identity,
    \begin{equation}\label{eqn:hypo_geo_id}
    \int \frac{ds}{\left((\delta-b)^2+s^2\right)^{m/2}} = \frac{s}{(\delta-b)^m}{}_2F_1\left(\frac{1}{2},\frac{m}{2};\frac{3}{2};-\frac{s^2}{(\delta-b)^2}\right).
    \end{equation}
    Exploiting the symmetry of the integrand over the real axis and using the hypergeometric identity to evaluate the improper integral, we obtain
    $$
    \lim_{N\rightarrow\infty}\int_{-N}^N \frac{ds}{\left((\delta-b)^2+s^2\right)^{m/2}} = \frac{2}{(\delta-b)^{m-1}} \lim_{N\rightarrow\infty} \left[\frac{N}{(\delta-b)}{}_2F_1\left(\frac{1}{2},\frac{m}{2};\frac{3}{2};-\frac{N^2}{(\delta-b)^2}\right)\right].
    $$
    Applying the asymptotic limit derived in~\cref{app:hypogeo_func}, we find that
    $$
    \int_{-\infty}^\infty \frac{ds}{\left((\delta-b)^2+s^2\right)^{m/2}} = \frac{2\Gamma\left(\frac{3}{2}\right)\Gamma\left(\frac{m-1}{2}\right)}{(\delta-b)^{m-1}\Gamma\left(\frac{m}{2}\right)}.
    $$
    The right-hand side achieves its maximum over $\delta/2<b<3\delta/2$ at the left-hand side of the interval. Evaluating there means that $\delta-b=\delta/2$ and substituting the result into the constant $C_f$ in~\cref{thm:uniform_int_bound} establishes the discretization error bound in~\cref{thm:trap_rule_1_error}.
\end{proof}

The derivation of the discretization error in~\cref{thm:trap_rule_1_error} is completely analogous, except that the strip of analyticity is taken with width $\sigma = \min\{a,\delta\}$.

\subsection{Truncation error}\label{sec:trunc_error}

Bounding the truncation error requires bounding the norms of the neglected tails when truncating the infinite series in~\cref{eqn:inf_sum}. Taking the norm of each term in the tail yields
$$
\|S_h(t;\delta)-S_{h,N}(t;\delta)\|\leq \frac{he^{\delta t}}{2\pi} \left[\sum_{|k|>N} |r(\delta+ihk)|\|R_A(\delta+ihk)\|\right]\|r(A)^{-1}x\|
$$
Applying the resolvent bound for semigroups in~\cref{thm:generation} (b), we find that
\begin{equation}\label{eqn:tail_bound}
\|S_h(t;\delta)-S_{h,N}(t;\delta)\|\leq \frac{Mhe^{\delta t}}{2\pi\delta} \left[\sum_{|k|>N} |r(\delta+ihk)|\right]\|r(A)^{-1}x\|
\end{equation}
Therefore, to bound the truncation error in the quadrature schemes of~\cref{sec:example_analysis,sec:high_order_scheme}, we need to bound the tail sum of the sampled regularizer.

\begin{proof}[Proof of truncation error in~\cref{thm:trap_rule_2_error}]
    To bound the tail sum, notice that $|r(\delta+is)|=(\delta^2+s^2)^{-m/2}$ decreases monotonically as $s$ increases. Therefore, we can bound the sum by integrating and, applying the hypergeometric identity in~\cref{eqn:hypo_geo_id}, we have
    \begin{align}\label{eqn:high_order_tail_bound}    
    &h\sum_{|k|>N} |r(\delta+ihk)| \leq 2\int_{hN}^\infty \frac{dx}{(\delta^2+s^2)^{m/2}} \\
    &= \frac{2}{\delta^{m-1}}\left[\frac{\Gamma\left(\frac{3}{2}\right)\Gamma\left(\frac{m-1}{2}\right)}{\Gamma\left(\frac{m}{2}\right)} - \frac{hN}{\delta}{}_2F_1\left(\frac{1}{2},\frac{m}{2};\frac{3}{2};-\left(\frac{hN}{\delta}\right)^2\right)\right].
    \end{align}
    Substituting into~\cref{eqn:tail_bound} establishes the truncation error bound in~\cref{thm:trap_rule_2_error}.
\end{proof}

Again, the truncation error in~\cref{thm:trap_rule_1_error} is derived in a completely analogous manner. The end result is slightly simpler due to the simplification of the constants and the reduction of the hypergeometric function with an $\arctan$ identity when $m=2$.

\section{Numerical experiments}\label{sec:num_exp}

In this section, we illustrate the computational framework with generators of the form
$$
A = F(x)\cdot \nabla_x, \qquad x\in \Omega\subset \mathbb{R}^d,
$$
where $F(x)$ is a Lipshitz continuous map $F:\mathbb{R}^d\rightarrow\mathbb{R}^d$ and $\Omega$ is a simply connected subset of $\mathbb{R}^d$. With suitable boundary conditions~\cite{ulmet1992properties}, $A$ generates a contraction semigroup on $C(\Omega)$ with the the usual supremum norm $\|g\|=\sup_{x\in\Omega}|g(x)|$, meaning that $A$ satisfies the hypotheses of~\cref{thm:generation} (a) with $M=1$ and $\omega=0$. Such generators may be highly non-normal and exhibit hallmark infinite-dimensional spectral properties of strongly continuous semigroups, such as point spectrum filling the entire left half-plane. On the other hand, they are useful for numerical experiments because analytic solutions may sometimes be obtained, e.g., through the method of characteristics. Moreover, they are intimately connected with Koopman operator theory for continuous-time dynamical systems and one can understand $K(t)g$ as the pullback of an observable $g\in C(\Omega)$ under the flow of the dynamical system $\dot x= F(x)$ on $\Omega$~\cite{budivsic2012applied}. 

\textbf{Example 1: Linear system with stable fixed point.} First, consider the linear system $\dot x = -x$ with $x\in\Omega=[-1,1]$. The analytic form of the flow is $\phi(x,t)=xe^{-t}$ and the pullback of an observable is given by the semigroup $[K(t)g](x)=g(xe^{-t})$ with generator $[Ag](x) = -xg'(x)$. No boundary conditions are required because $F(x)=-x$ points inward at the boundary points $x=\pm 1$~\cite{ulmet1992properties}. \Cref{fig:example_1} displays the output of the algorithm outlined in~\cref{sec:high_order_scheme} at times $t=0.2,0.4,0.6,0.8,1.0$ in the left-hand panel with order $m=6$, contour location $\delta=2$, and $N=80$ quadrature nodes. The relative errors in the computed approximations of $K(t)g$ are compared with the error bounds from~\cref{thm:trap_rule_2_error} for $0\leq t\leq 1$ in the middle panel and the numerical convergence of the scheme for $m=2,4,6,8$ with $N$ is demonstrated in the right panel. The test-case observable is $g(x) = \sin(\pi x)(1-x^2)$ and~\cref{fig:example_1} (left) shows the expected flattening of the observable around the stable fixed point of the dynamical system as $t$ increases.

\begin{figure}
    \centering
    \begin{minipage}{0.31\textwidth}
        \begin{overpic}[width=\textwidth]{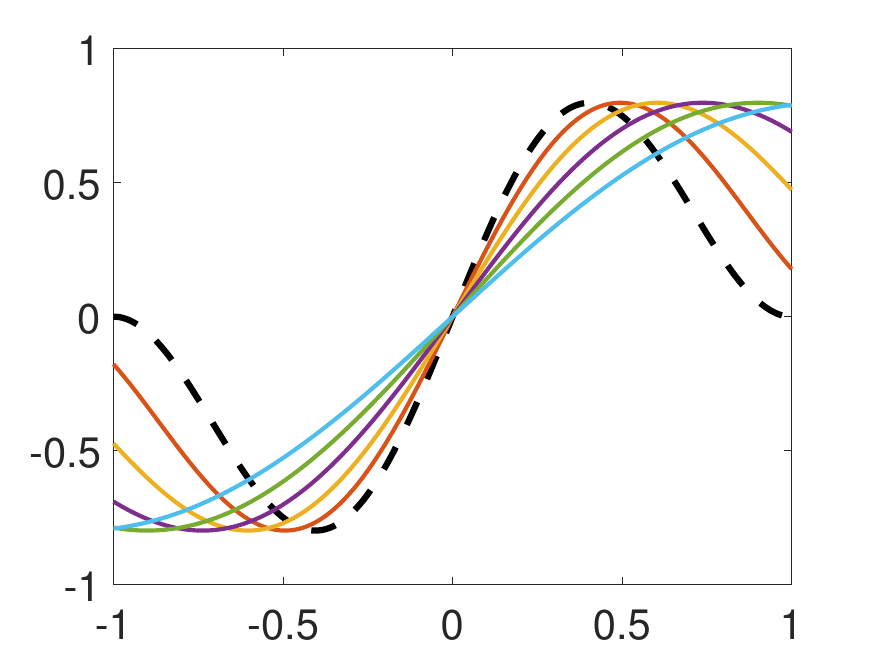}
        \put (44,-1) {$\displaystyle x$}
        \put (36,71) {$\displaystyle \hat K^{(6)}(t)g(x)$}
        \end{overpic}
    \end{minipage}
    \begin{minipage}{0.31\textwidth}
        \begin{overpic}[width=\textwidth]{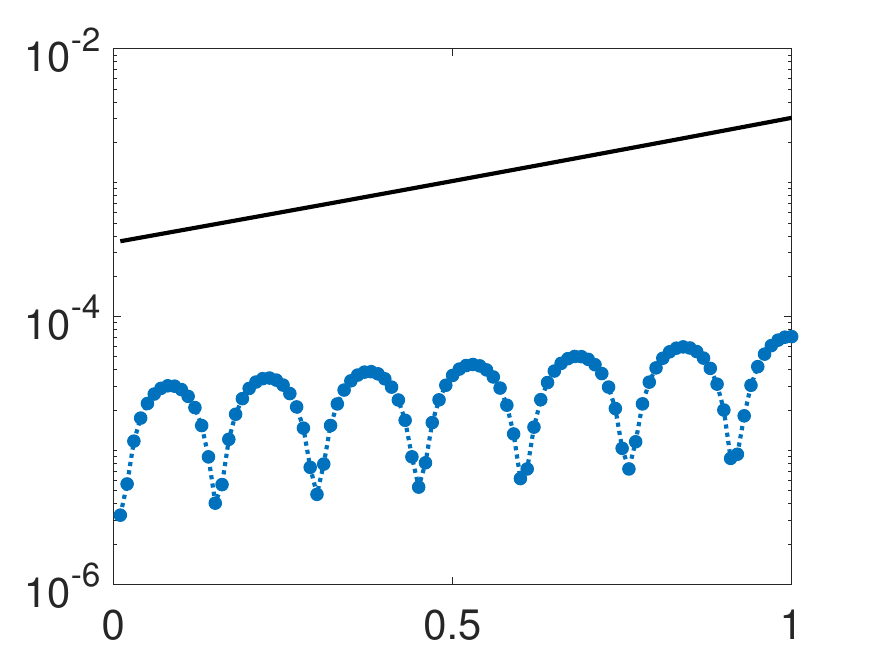}
        \put (44,-1) {$\displaystyle t$}
        \put (14,72) {$\displaystyle \|K(t)g-\hat K^{(6)}(t)g\|$}
        \end{overpic}
    \end{minipage}
    \begin{minipage}{0.31\textwidth}
        \begin{overpic}[width=\textwidth]{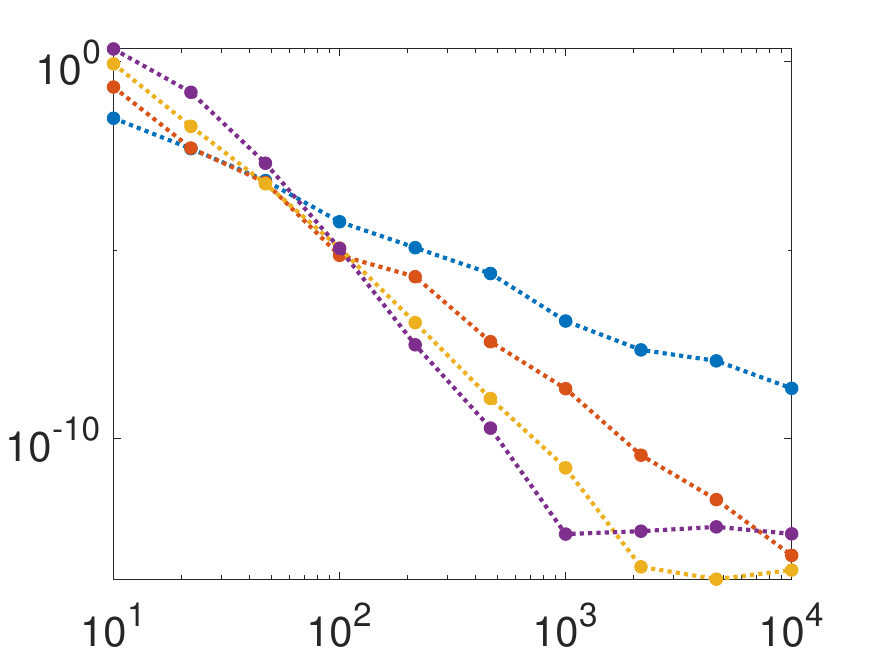}
        \put (48,-1) {$\displaystyle N$}
        \put (14,72) {$\displaystyle \|K(1)g-\hat K^{(m)}(1)g\|$}
        \end{overpic}
    \end{minipage}
    \caption{\label{fig:example_1} \textbf{Example 1.} In the left-hand panel, the computed functions $\hat K^{(6)}(t) g$ with observable $g(x)=\sin(\pi x)(1-x^2)$ (dashed black line) is shown at $t=0.2$ (red), $0.4$ (yellow), $0.6$ (purple), $0.8$ (green), and $1.0$ (cyan). The middle panel compares the computed error (dashed line) with the error bounds from~\cref{thm:trap_rule_2_error} for $0\leq t\leq 1$, $\delta=2$, $N=80$, and $h$ chosen to minimize the error bound at $t=1$. In the right-hand panel, the error in the approximation $\hat K^{(m)}(1)$ at $t=1$ is plotted against $N$ for $m=2$ (blue), $4$ (red), $6$ (yellow), $8$ (purple).}
\end{figure}

\begin{figure}
    \centering
    \begin{minipage}{0.31\textwidth}
        \begin{overpic}[width=\textwidth]{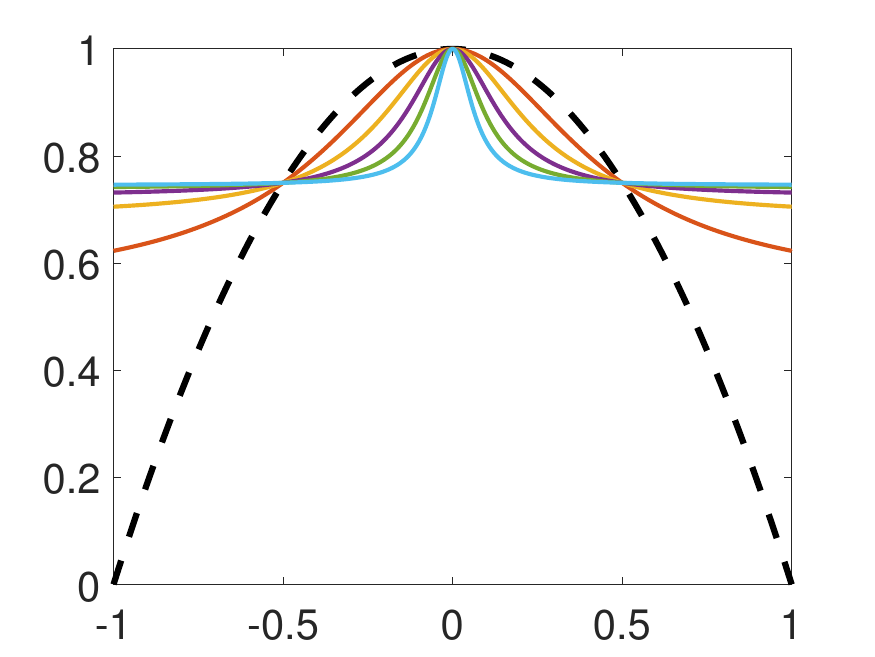}
        \put (44,-1) {$\displaystyle x$}
        \put (36,71) {$\displaystyle \hat K^{(6)}(t)g(x)$}
        \end{overpic}
    \end{minipage}
    \begin{minipage}{0.31\textwidth}
        \begin{overpic}[width=\textwidth]{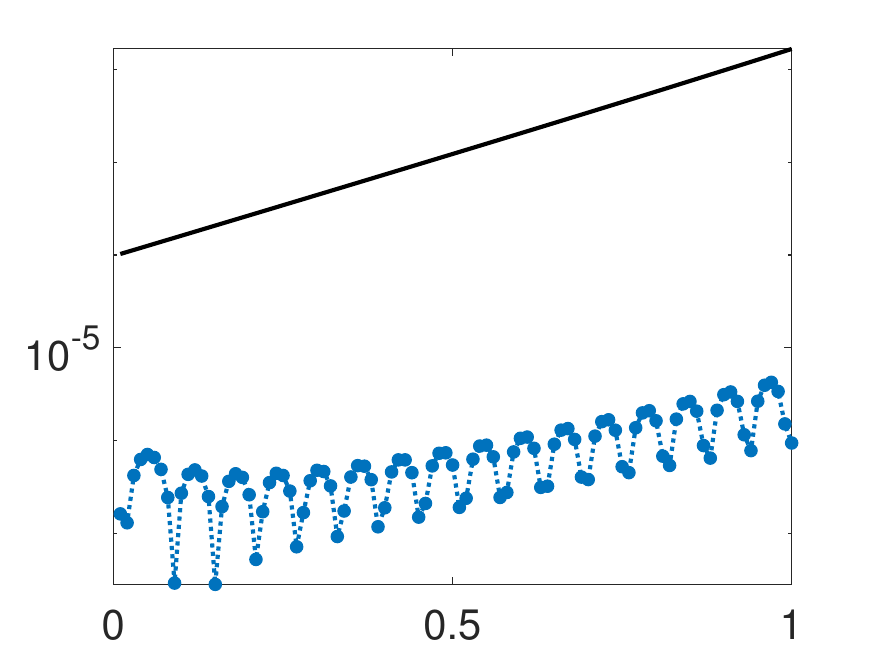}
        \put (44,-1) {$\displaystyle t$}
        \put (14,72) {$\displaystyle \|K(t)g-\hat K^{(6)}(t)g\|$}
        \end{overpic}
    \end{minipage}
    \begin{minipage}{0.31\textwidth}
        \begin{overpic}[width=\textwidth]{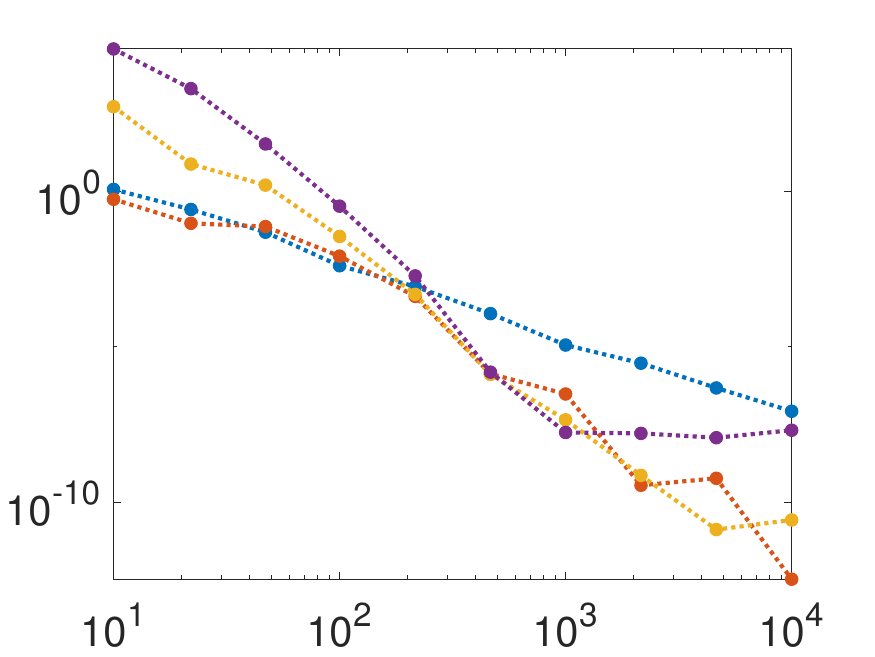}
        \put (48,-1) {$\displaystyle N$}
        \put (14,72) {$\displaystyle \|K(1)g-\hat K^{(m)}(1)g\|$}
        \end{overpic}
    \end{minipage}
    \caption{\label{fig:example_2} \textbf{Example 2.} In the left-hand panel, the computed functions $\hat K^{(6)}(t) g$ with observable $g(x)=\sin(\pi x)(1-x^2)$ (solid black line) is shown at $t=0.2$ (red), $0.4$ (yellow), $0.6$ (purple), $0.8$ (green), and $1.0$ (cyan). The middle panel compares the computed error (dashed line) with the error bounds from~\cref{thm:trap_rule_2_error} for $0\leq t\leq 1$, $\delta=5$, $N=500$, and $h$ chosen to minimize the error bound at $t=1$. In the right-hand panel, the error in the approximation $\hat K^{(m)}(t)g$ is plotted against $N$ for $m=2$ (blue), $4$ (red), $6$ (yellow), $8$ (purple).}
\end{figure}
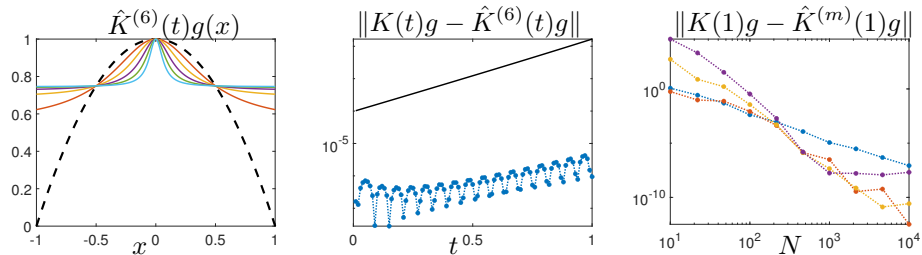

\textbf{Example 2: Nonlinear system with three fixed points.} Next, consider the nonlinear system $\dot x = 2x-8x^3$ with $x\in\Omega=[-1,1]$. The analytic form of the flow is $\phi(x,t)=xe^{2t}/\sqrt{1-4x^2(1+e^{4t})}$ and the pullback of an observable is given by the semigroup $[K(t)g](x)=g(\phi(x,t))$ with generator $[Ag](x) = (2x-8x^3)g'(x)$. Again, no boundary conditions are required because $F(x)=2x-8x^3$ points inward at the boundary points $x=\pm 1$. \Cref{fig:example_2} displays the output of the algorithm outlined in~\cref{sec:high_order_scheme} at times $t=0.2,0.4,0.6,0.8,1.0$ in the left panel for $m=6$, $\delta=5$, and $N=500$. The relative errors in the computed approximations of $K(t)g$ are compared with the error bounds from~\cref{thm:trap_rule_2_error} for $0\leq t\leq 1$ in the middle panel and the numerical convergence of the scheme for $m=2,4,6,8$ with $N$ is demonstrated in the right panel. The test-case observable is $g(x) = 1-x^2$ and~\cref{fig:example_2} (left) shows the expected flattening of the observable around the stable fixed points at $x=\pm 1/2$ and localization around the unstable fixed point at $x=0$ as $t$ increases.

\textbf{Example 3: 2D linear system.} The third example is the linear $2$D coupled system of oscillators given by $\dot x = Bx$ with $B=[0 \quad 1 \,;\, -1\quad  0]$. The analytic form of the flow is $\phi(x,t)=\exp(Bt)x$ and the associated semigroup action is $[K(t)g](x)=g(\phi(x,t))$. The semigroup's generator is $[Ag](x) = x_2\partial_1 g(x) -x_1\partial_2 g(x)$ and we require that $g$ be compactly supported on $\mathbb{R}^2$~\cite[pp.~91-92]{engel2000one}. \Cref{fig:example_3} displays the output of the algorithm outlined in~\cref{sec:high_order_scheme} at times $t=0$ (left), $t=1$ (middle), and $t=2$ (right) for $m=10$, $\delta = 4$, and $N=194$. The test-case observable is $g(x) = \exp(-2x_1^2-0.5x_2^2)$ and~\cref{fig:example_3} illustrates the expected rotation of the observable based on the underlying rotation of the state space by the dynamical system $\dot x = Bx$. For computation, the resolvent was discretized with a second-order difference stencil on a $201\times 201$  equispaced grid and the maximum error on the grid at $t=2$ was $0.004$, compared with a computed error bound of $0.0079$ from~\cref{thm:trap_rule_2_error}. 

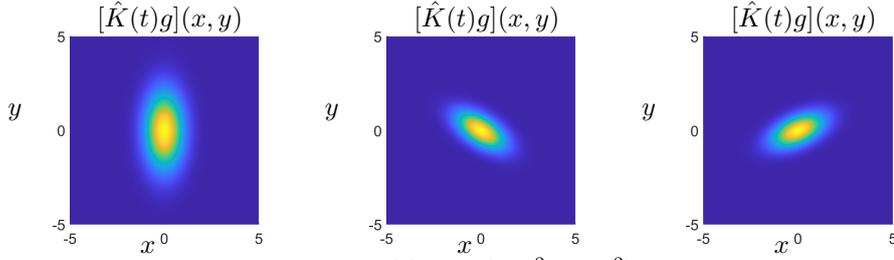
\begin{figure}
    \centering
    \begin{minipage}{0.31\textwidth}
        \begin{overpic}[width=\textwidth]{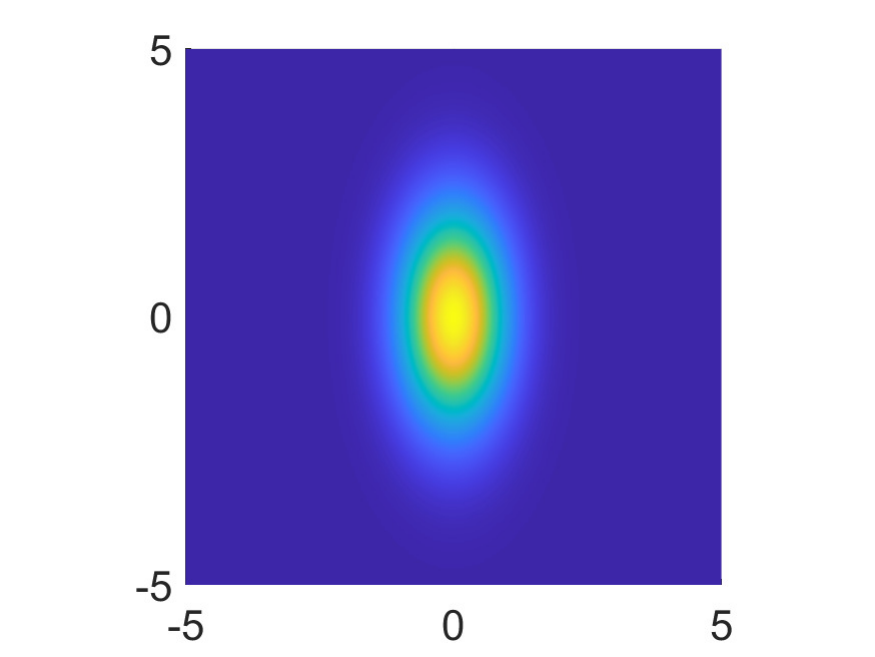}
        \put (44,-1) {$\displaystyle x$}
        \put (1,44) {$\displaystyle y$}
        \put (30,73) {$\displaystyle [\hat K(t)g](x,y)$}
        \end{overpic}
    \end{minipage}
    \begin{minipage}{0.31\textwidth}
        \begin{overpic}[width=\textwidth]{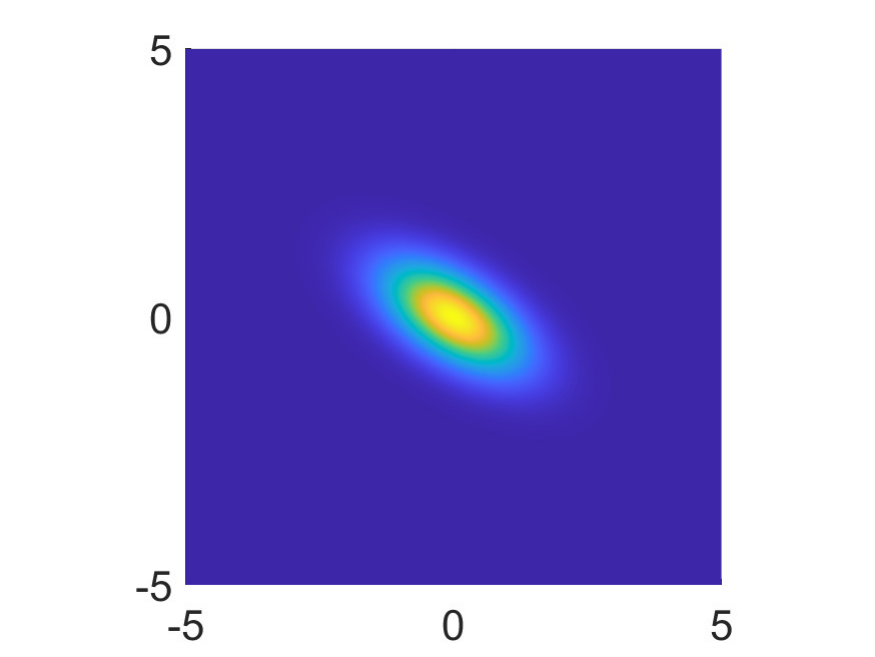}
        \put (44,-1) {$\displaystyle x$}
        \put (1,44) {$\displaystyle y$}
        \put (30,73) {$\displaystyle [\hat K(t)g](x,y)$}
        \end{overpic}
    \end{minipage}
    \begin{minipage}{0.31\textwidth}
        \begin{overpic}[width=\textwidth]{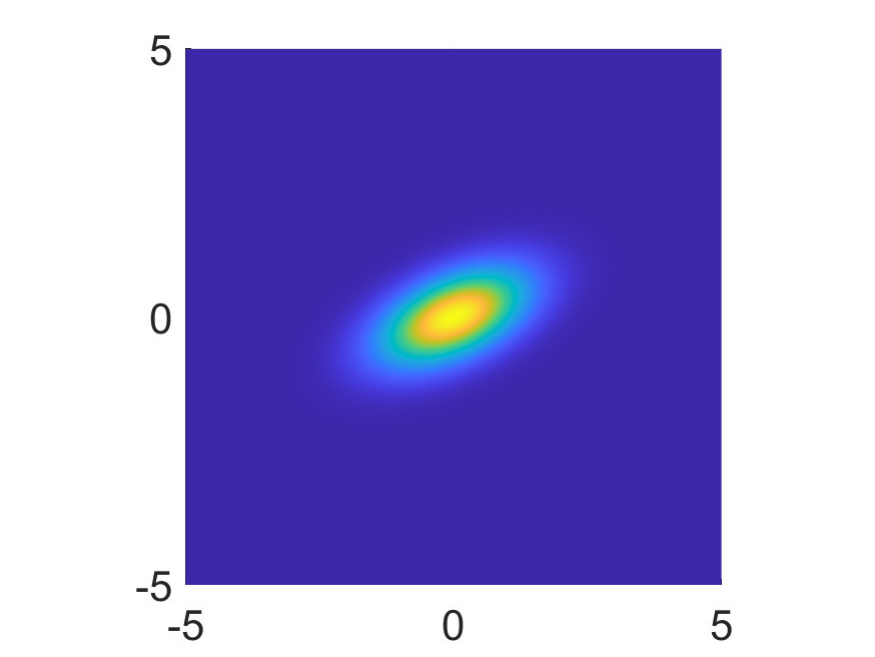}
        \put (44,-1) {$\displaystyle x$}
        \put (1,44) {$\displaystyle y$}
        \put (30,73) {$\displaystyle  [\hat K(t)g](x,y)$}
        \end{overpic}
    \end{minipage}
    \caption{\label{fig:example_3} \textbf{Example 3.} The evolution of $g(x)=\exp(-2x^2-0.5y^2)$ under the action of the semigroup in Example 3 is plotted at $t=0$ (left), $t=1$ (middle), and $t=2$ (right), demonstrating the expected counterclockwise rotation of the initial condition on the underlying state-space.}
\end{figure}

\textbf{Example 4: 2D nonlinear system.} The fourth and final example is the nonlinear separable $2$D system given by $\dot x = 2x-8x^3$ and $\dot y = 2y-8y^3$. The analytic form of the flow is a tensor product copy of the $1$D solution in Example 2 and the associated semigroup's action is $[K(t)g](x)=g(\phi(x,t))$ with generator $[Ag](x) = (2x-8x^3)\partial_x g(x,y) + (2y-8y^3)\partial_y g(x)$. \Cref{fig:example_4} displays the output of the algorithm outlined in~\cref{sec:high_order_scheme} at times $t=0$ (left), $t=0.1$ (middle), and $t=0.2$ (right) for $m=4$, $\delta=16$, and $N=97$. The test-case observable is $g(x,y) = \exp(-2x^2-0.5y^2)$ and~\cref{fig:example_4} shows the expected shrinking and stretching of the observable around the underlying fixed points of the separable dynamical system. For computation, the resolvent was discretized with a second-order difference stencil on a $251\times 251$  equispaced grid and the maximum error on the grid at $t=2$ was $0.0057$, compared with a computed error bound of $0.0076$ from~\cref{thm:trap_rule_2_error}. 

\begin{figure}
    \centering
    \begin{minipage}{0.31\textwidth}
        \begin{overpic}[width=\textwidth]{figures/example4_slice0.pdf}
        \put (44,-1) {$\displaystyle x$}
        \put (1,44) {$\displaystyle y$}
        \put (30,73) {$\displaystyle [\hat K(t)g](x,y)$}
        \end{overpic}
    \end{minipage}
    \begin{minipage}{0.31\textwidth}
        \begin{overpic}[width=\textwidth]{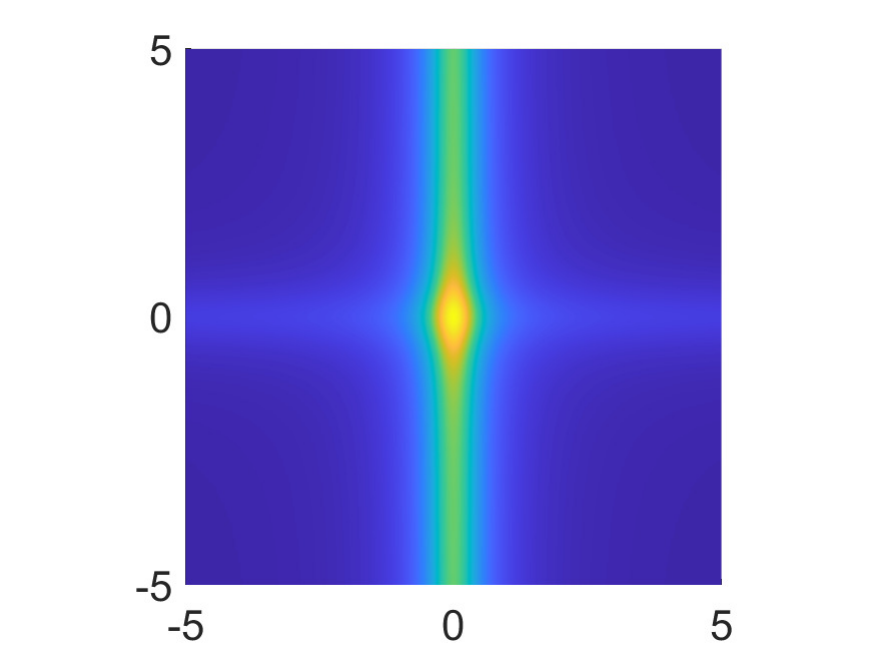}
        \put (44,-1) {$\displaystyle x$}
        \put (1,44) {$\displaystyle y$}
        \put (30,73) {$\displaystyle [\hat K(t)g](x,y)$}
        \end{overpic}
    \end{minipage}
    \begin{minipage}{0.31\textwidth}
        \begin{overpic}[width=\textwidth]{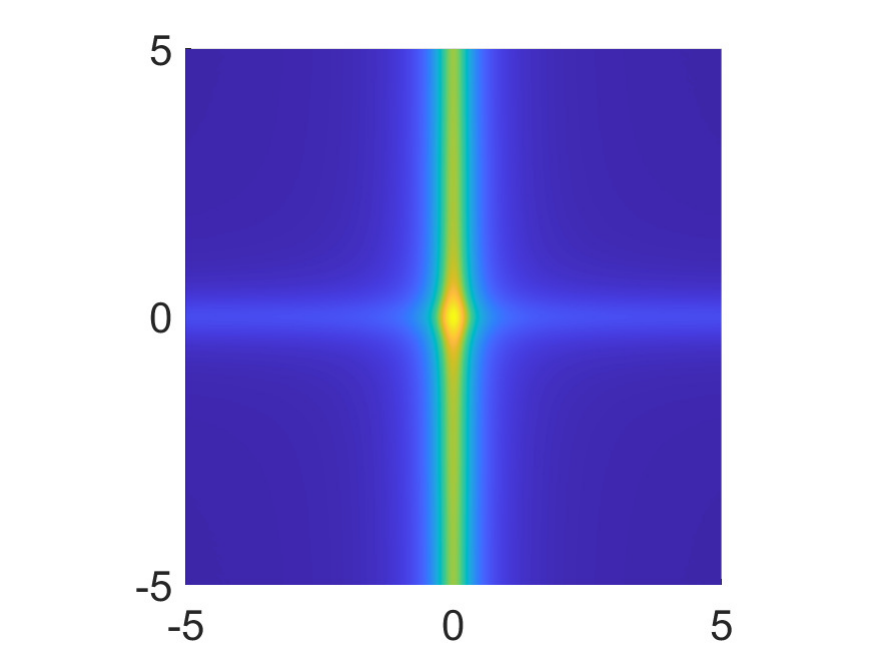}
        \put (44,-1) {$\displaystyle x$}
        \put (1,44) {$\displaystyle y$}
        \put (30,73) {$\displaystyle  [\hat K(t)g](x,y)$}
        \end{overpic}
    \end{minipage}
    \caption{\label{fig:example_4} \textbf{Example 4.} The evolution of $g(x)=\exp(-2x^2-0.5y^2)$ under the action of the semigroup in Example 4 is plotted at $t=0$ (left), $t=0.1$ (middle), and $t=0.2$ (right). Note the concentration along the axes caused by the roots of the variable coefficients in the first-order generator.}
\end{figure}
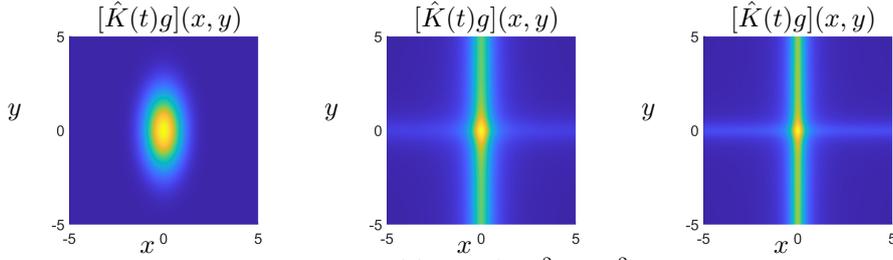

\section*{Acknowledgments}
The authors would like to thank Alex Townsend and Matthew Colbrook for helpful comments on an earlier version of this manuscript. They would also like to acknowledge the support of AFOSR Laboratory Task programs \#21RQCOR083 and \#24RQCOR001. These tasks are jointly funded by Drs.~Fariba Fahroo, Frederick Leve, and Warren Adams. Finally, they thank Dr.~Andrew Leonard, whose early investigation into the potential of contour integral methods for Koopman semigroups under AFOSR Laboratory Task program \#21RQCOR083 led to this work.

\begin{appendices}
\crefalias{section}{appendix}

\section{Quadrature errors in a Banach space}\label{app:quad_banach}

Here, we establish a straightforward extension of the discretization error bound for the trapezoidal rule~\cite[Thm.~5.1]{trefethen2014exponentially} to the setting of Banach-valued functions $f:\mathbb{C}\rightarrow\mathcal{X}$. 
\begin{theorem}\label{appthm:trap_rule_error}
    Suppose $f:\mathbb{C}\rightarrow\mathcal{X}$ is analytic in the strip $|{\rm Im}\,z|<a$ for some $a>0$, that $f(z)\rightarrow 0$ uniformly as $|z|\rightarrow\infty$ in the strip, and for some $C_f>0$, it satisfies
    \begin{equation}\label{appeqn:uniform_int_bound}
        \int_{-\infty}^\infty \|f(s+ib)\|\,ds\leq C_f, \qquad\text{for all}\qquad -a < b < a.
    \end{equation}
    If, for some $h>0$, the sum $h\sum_{k=-\infty}^\infty f(hk)$ converges absolutely, then it satisfies
    \begin{equation}\label{appeqn:trap_rule_bound}
        \left\|\int_{\infty}^\infty f(s)\,ds - h\sum_{k=-\infty}^\infty f(hk)\right\|\leq \frac{2C_f}{e^{2\pi a/h}-1}.
    \end{equation}
    Moreover, the constant $2C_f$ in the numerator is sharp, i.e., as small as possible.
\end{theorem}
\begin{proof}
    The assumption of analyticity and uniform integrability of $f$ in the strip $-a<b<a$, together with separability of $\mathcal{X}$, means that $f$ is Bochner integrable. Consequently, it is permissible to exchange linear functionals (and, generally, operators) on $\mathcal{X}$ with the integral sign. Taking an arbitrary linear functional $x'\in\mathcal{X}'$ with dual norm $\|x'\|_{\mathcal{X}'}=1$ and, denoting the duality pairing by $\langle\cdot,\cdot\rangle$, we compute
    $$
    \left\langle x', \int_{\infty}^\infty f(s)\,ds - h\sum_{k=-\infty}^\infty f(hk)\right\rangle = \int_{\infty}^\infty \left\langle x',f(s)\right\rangle  \,ds - h\sum_{k=-\infty}^\infty \left\langle x', f(hk)\right\rangle.
    $$
    Note that the infinite series is absolutely and uniformly convergent, permitting the interchange for the sum as well the integral. The scalar function $\langle x',f(s)\rangle$ inherits the analyticity and integrability of $f$ in the strip with the same constant of integrability $C_f$. Applying the scalar trapezoidal error bound for the discretization error on the right side and taking supremums over the unit ball, $x'\in \mathcal{X}'$ with $\|x'\|_{\mathcal{X}'}\leq 1$, on both sides of the equality above provides the desired bound in~\cref{appeqn:trap_rule_bound}.
\end{proof}

\section{Hypergeometric functions}\label{app:hypogeo_func}

To develop rigorous and explicit bounds for the error in our high-order quadrature approximations, we make extensive use of the hypergeometric identity
\begin{equation}\label{appeqn:hypo_geo_id}
\int \frac{ds}{\left((\delta-b)^2+s^2\right)^{m/2}} = \frac{s}{(\delta-b)^m}{}_2F_1\left(\frac{1}{2},\frac{m}{2};\frac{3}{2};-\frac{s^2}{(\delta-b)^2}\right).
\end{equation}
This appendix summarizes a few useful asymptotic properties and identities that allow us to compute the definite integrals that appear in the high-order quadrature bounds.

We first calculate the asymptotics of ${}_2F_1(a,b;c;z)$ as $z\rightarrow-\infty$, following the suggestion in the Digital Library of Mathematical Functions~\cite[\href{https://dlmf.nist.gov/15.12.i}{15.12(i)}]{NIST:DLMF}. We start with the series representation of the normalized hypergeometric function~\cite[\href{https://dlmf.nist.gov/15.2.E2}{15.2.2}]{NIST:DLMF}
\begin{equation}\label{appeqn:gamma_series}
    \tilde F(a,b;c;z) = \Gamma(c)^{-1}{}_2F_1(a,b;c;z) = \sum_{k=0}^\infty \frac{(a)_k(b)_k}{\Gamma(c+k) k!}z^k, \qquad\text{for}\qquad |z|<1.
\end{equation}
Here, $(\cdot)_k$ is Pochhammer's symbol. Applying the identity in~\cite[\href{https://dlmf.nist.gov/15.8.E2}{15.8.2}]{NIST:DLMF}, we have that
\begin{align}\label{appeqn:hypogeo_flip}
    \frac{\sin(\pi(b-a))}{\pi}\tilde F(a,b;c;z) = &\frac{(-z)^a}{\Gamma(b)\Gamma(c-a)}\tilde F\left(a,a-c+1;a-b+1;\frac{1}{z}\right) \\ &- \frac{(-z)^b}{\Gamma(a)\Gamma(c-b)}\tilde F\left(b,b-c+1;b-a+1;\frac{1}{z}\right).
\end{align}
Substituting the series representation from~\cref{appeqn:gamma_series} on the right-hand side and taking the limit $z\rightarrow-\infty$, we find that the leading order expansion (for $0<a<b$) is
\begin{equation}\label{appeqn:hypogeo_limit}
    \tilde F(a,b;c;z) \sim \left[\frac{\pi}{\sin(\pi(b-a))\Gamma(b)\Gamma(c-a)\Gamma(a-b+1)}\right]\frac{1}{(-z)^a}.
\end{equation}
Multiplying both  sides by $(-z)^a\Gamma(c)$ and substituting the relevant values of $a=1/2$, $b=m/2$, $c=3/2$, and $z=-s^2/(\delta-b)^2$, we determine that
$$
\lim_{s\rightarrow\infty} \left(\frac{s}{\delta-b}\right){}_2F_1\left(\frac{1}{2},\frac{m}{2};\frac{3}{2};-\frac{s^2}{(\delta-b)^2}\right) = \frac{\pi}{\sin(\pi(m-1)/2)\Gamma(m/2)}\frac{\Gamma(3/2)}{\Gamma((3-m)/2)}.
$$
Finally, we simplify the right-hand side by applying Euler's reflection formula,
$$
\frac{\pi\Gamma(3/2)}{\sin(\pi(m-1)/2)\Gamma(m/2)\Gamma((3-m)/2)} = \frac{\Gamma(3/2)\Gamma((m-1)/2)}{\Gamma(m/2)}.
$$
The last two results combine with~\cref{appeqn:hypo_geo_id} to establish the definite integral calculated in the proof of the discretization error in~\cref{sec:disc_error}.  

Next, we consider the rate at which $z{}_2F_1(a,b;c;z)$ approaches its limit as $z\rightarrow-\infty$ for fixed values of $a=1/2$, $b=m/2$, and $c=3/2$. At these values, the series expansion in~\cref{appeqn:gamma_series} of the first term on the right-hand side of~\cref{appeqn:hypogeo_flip} contains only the leading order term because $a-c+1=\frac{1}{2}-\frac{3}{2}+1 = 0$ and $(0)_k=0$ for all $k\geq 1$ (while $(0)_0=1$). The higher order terms are contributed by the series expansion of the second term in~\cref{appeqn:hypogeo_flip} and we find that, as $z\rightarrow -\infty$, the rate of approach to the limit is
$$
\frac{\Gamma(3/2)\Gamma((m-1)/2)}{\Gamma(m/2)} - (-z)^{1/2}{}_2F_1\left(\frac{1}{2},\frac{m}{2};\frac{3}{2};z\right) = \frac{1}{m-1}\left(z\right)^{(m-1)/2} + \mathcal{O}\left((z)^{(m+1)/2}\right).
$$
Plugging in $z = -(hN)^2/(\delta-b)^2$ establishes the leading order asymptotic in~\cref{eqn:asymptotic_truncation_bound}.

Finally, we recall that the hypergeometric function can be expressed in terms of an $\arctan$ function when $m=2$. We have that
$$
\sqrt{z}{}_2F_1\left(\frac{1}{2},1;\frac{3}{2},z\right) = \arctan(\sqrt{z}).
$$
This allows us to compare the bounds on the truncation error in~\cref{thm:trap_rule_1_error} and~\cref{thm:trap_rule_2_error} when $m=2$ and establishes their equivalence when $a=\delta$ in~\cref{sec:example_analysis}.

\end{appendices}


\bibliography{sn-bibliography}

\end{document}